\documentclass[10pt,reqno]{amsart}
\usepackage{amsmath, amsthm, amscd, amsfonts, amssymb, graphicx, color, mathrsfs}
\usepackage{lineno}
\usepackage[bookmarksnumbered, colorlinks, plainpages]{hyperref}
\hypersetup{colorlinks=true,linkcolor=red, anchorcolor=green, citecolor=cyan, urlcolor=red, filecolor=magenta, pdftoolbar=true}
\modulolinenumbers[5]

\newtheorem{thm}{Theorem}[section]
\newtheorem{lemma}[thm]{Lemma}
\newtheorem{pro}[thm]{Proposition}
\newtheorem{cor}[thm]{Corollary}
\theoremstyle{definition}
\newtheorem{defn}[thm]{Definition}

\newtheorem{hyp}[thm]{Hypothesis}

\newtheorem{remark}[thm]{Remark}
\theoremstyle{remark}

\newtheorem*{ack}{Acknowledgements}
\numberwithin{equation}{section}

\newcommand{\ol}[1]{\overline{#1}}

\renewcommand{\hat}[1]{\widehat{#1}}
\renewcommand{\tilde}[1]{\widetilde{#1}}

\newcommand{\set}[1]{{\left\{#1\right\}}}
\newcommand{\pa}[1]{{\left(#1\right)}}
\newcommand{\sq}[1]{{\left[#1\right]}}
\newcommand{\gen}[1]{{\left\langle #1\right\rangle}}
\newcommand{\abs}[1]{{\left|#1\right|}}
\newcommand{\norm}[1]{{\left\|#1\right\|}}

\newcommand{\ssm}{\smallsetminus}

\newcommand{\ra}{\rightarrow}

\newcommand{\longra}{\longrightarrow}

\newcommand{\N}{\mathbb{N}}

\newcommand{\R}{\mathbb{R}}

\newcommand{\esssup}{\operatorname{ess\,sup}}
\newcommand{\eqsys}[1]{{\left\{\begin{array}{ll}#1\end{array}\right.}}
\newcommand{\elle}{\operatorname{L}}

\newcommand{\tc}{\, \middle |\,}                                                

\newcommand{\con}{\operatorname{\mathscr{C}}}

\newcommand{\eps}{\varepsilon}
\newcommand{\ball}{\operatorname{\mathscr{B}}}
\newcommand{\sfera}{\operatorname{\mathcal{S}}}

\newcommand{\fcon}{\operatorname{\mathscr{FC}}}

\DeclareMathOperator{\linspan}{\operatorname{span}}

\DeclareMathOperator{\lip}{\operatorname{Lip}}
\DeclareMathOperator{\diver}{\operatorname{div}}

\DeclareMathOperator{\trace}{\operatorname{Tr}}
\DeclareMathOperator{\dom}{\operatorname{dom}}

\begin{document}

\frenchspacing

\title[On the domain of elliptic operators defined in subsets of Wiener spaces]{On the domain of elliptic operators defined in subsets of Wiener spaces}

\author[D. Addona]{{D. Addona}}

\address[D. Addona]{Dipartimento di Matematica e Informatica, Universit\`a di Ferrara, via Machiavelli, 35, 44121 Ferrara, Italy}
\email{\textcolor[rgb]{0.00,0.00,0.84}{d.addona@campus.unimib.it}}

\author[G. Cappa]{{G. Cappa}}

\address[G. Cappa]{Dipartimento di Scienze Matematiche, Fisiche e Informatiche, Universit\`a degli Studi di Parma, Parco Area delle Scienze 53/A, 43124 Parma, Italy}
\email{\textcolor[rgb]{0.00,0.00,0.84}{gianluca.cappa@nemo.unipr.it}}

\author[S. Ferrari]{{S. Ferrari}$^*$}

\address[S. Ferrari]{Dipartimento di Matematica e Fisica ``Ennio de Giorgi'', Universit\`a del Salento, Via per Arnesano snc, 73100 Lecce, Italy.}
\email{\textcolor[rgb]{0.00,0.00,0.84}{simone.ferrari@unisalento.it}}


\begin{abstract}
Let $X$ be a separable Banach space endowed with a non-degenerate centered Gaussian measure $\mu$. The associated Cameron--Martin space is denoted by $H$. Consider two sufficiently regular convex functions $U:X\ra\R$ and $G:X\ra \R$. We let $\nu=e^{-U}\mu$ and $\Omega=G^{-1}(-\infty,0]$. In this paper we study the domain of the the self-adjoint operator associated with the quadratic form
\begin{gather}\label{inabstract}
(\psi,\varphi)\mapsto \int_\Omega\gen{\nabla_H\psi,\nabla_H\varphi}_Hd\nu\qquad\psi,\varphi\in W^{1,2}(\Omega,\nu),
\end{gather}
and we give sharp embedding results for it. In particular we obtain a characterization of the domain of the Ornstein--Uhlenbeck operator on half-spaces, namely if $U\equiv 0$ and $G$ is an affine function, then the domain of the operator defined via \eqref{inabstract} is the space
\[\set{u\in W^{2,2}(\Omega,\mu)\tc \gen{\nabla_H u(x),\nabla_H G(x)}_H=0\text{ for }\rho\text{-a.e. }x\in G^{-1}(0)},\]
where $\rho$ is the Feyel--de La Pradelle Hausdorff--Gauss surface measure.
\end{abstract}



\date{\today
\newline \indent $^{*}$ Corresponding author}

\maketitle

\linenumbers

\section{Introduction}
Let $X$ be a separable Banach space with norm $\norm{\cdot}_X$, endowed with a non-degenerate centered Gaussian measure $\mu$. The associated Cameron--Martin space is denoted by $H$, its inner product by $\gen{\cdot,\cdot}_H$ and its norm by $\abs{\cdot}_H$.
The spaces $W^{k,p}(X,\mu)$ for $p\geq 1$ and $k\in\N$ are the classical Sobolev spaces of the Malliavin calculus (see \cite[Chapter 5]{Bog98}).

The aim of this paper is to study the domain of the self-adjoint operator $L_{\nu,\Omega}$ associated with the quadratic form
\[(\psi,\varphi)\mapsto \int_\Omega\gen{\nabla_H\psi,\nabla_H\varphi}_Hd\nu\qquad\psi,\varphi\in W^{1,2}(\Omega,\nu),\]
where $\Omega$ is a convex subset of $X$, $\nu:=e^{-U}\mu$ and $U:X\ra\R$ is a convex function, $\nabla_H \psi$ is the gradient along $H$ of $\psi$ and $W^{1,2}(\Omega,\nu)$ is the Sobolev space on $\Omega$ associated to the measure $\nu$ (see Section \ref{Notations and preliminaries}). These operators arise in Kolmogorov equations in Hilbert spaces corresponding to stochastic variational inequalities with reflection, such as
\[\eqsys{dY(t,x)-Y(t,x)dt-\nabla_H U(Y(t,x))dt+N_\Omega(Y(t,x))dt\ni dW(t,x);\\ Y(0,x)=x,}\]
where $N_\Omega$ is the normal cone to $\Omega$ and $W(t,\cdot)$ is a $X$-valued cylindrical Wiener process (here $X$ is a Hilbert space). This is because, at least formally, the transition semigroup $T(t)f(x):=\mathbb E[f(Y(t,x))]$ is generated by $L_{\nu,\Omega}$.

In the case of the standard Gaussian measure in a convex subset $\Omega\subseteq\R^n$ with sufficiently regular boundary, the operator $L_{\nu,\Omega}$ reads as
\[L_{\nu,\Omega} u(\xi)=\Delta u(\xi)-\gen{\nabla U(\xi)+\xi,\nabla u(\xi)}\qquad u\in\con^2_b(\Omega),\]
so that, if $U$ is sufficiently regular, $L_{\nu,\Omega}$ is an elliptic operator with possibly unbounded coefficients, and its domain in $\elle^2(\Omega,\nu)$ is
\begin{gather*}
D(L_{\nu,\Omega})=\eqsys{\set{u\in W^{2,2}(\R^n,\nu)\tc \gen{\nabla U+\xi,\nabla u}\in \elle^2(\R^n,\nu)}, & \Omega=\R^n;\\ \set{u\in W^{2,2}(\Omega,\nu)\tc \gen{\nabla U+\xi,\nabla u}\in \elle^2(\Omega,\nu),\ \partial u/\partial n=0\text{ at } \partial\Omega}, & \Omega\neq\R^n,}
\end{gather*}
where $\partial/\partial n$ is the exterior normal derivative at the boundary of $\partial\Omega$ (see \cite{DPL04} and \cite{LL06}). In the infinite dimensional case there is a characterization for the Ornstein--Uhlenbeck operator, when $\Omega$ is the whole space and $U\equiv 0$ (see \cite[Section 5.6]{Bog98}). In this case the operator $L_{\mu,X}$ is the infinitesimal generator of the Ornstein--Uhlenbeck semigroup
\[T_t f(x)=\int_Xf\pa{e^{-t}x+\sqrt{1-e^{-2t}}y}d\mu(y),\]
in $\elle^2(X,\mu)$ and its domain is $D(L_{\mu,X})=W^{2,2}(X,\mu)$. Further results were obtained in \cite{CF16}, assuming $U$ has $H$-Lipschitz gradient, and $\Omega$ is the whole space. In this case too the domain is $D(L_{\nu,X})=W^{2,2}(X,\nu)$. We want to point out that in \cite{MV08} the authors study in detail the case of non-symmetric Ornstein--Uhlenbeck operators on the whole space.

This paper is a first attempt to give a characterization of the domain of $L_{\nu,\Omega}$ in a more general setting. In order to state the main results of this paper we need some hypotheses on the set $\Omega$ and on the weighted measure $\nu$.

Throughout the paper we take $\Omega=G^{-1}(-\infty,0]$, where $G$ satisfies the following assumptions.
\begin{hyp}\label{ipotesi dominio}
Let $G:X\ra\R$ be a version of a function belonging to $W^{3,q}(X,\mu)$ for every $q>1$. We fix a version of $\nabla_H G$ and a version of $\nabla_H^2 G$ such that
\begin{enumerate}
\item $G$ is convex and, for every $q>1$, the functions $G$ is $(3,q)$-precise (see Section \ref{Special classes of functions});\label{ipo dominio convessita}

\item for every $q>1$, the functions $\nabla_H G$ and $\nabla_H^2 G$ are $(2,q)$-precise and $(1,q)$-precise, respectively (see Section \ref{Special classes of functions});\label{ipo dominio precisione}

\item $\mu(G^{-1}(-\infty,0])>0$ and $G^{-1}(-\infty,0]$ is closed;\label{ipo dominio convessita e chiusura}

\item $\abs{\nabla_H G}_H^{-1}\in\elle^q(G^{-1}(-\infty,0],\mu)$ for every $q>1$;\label{ipo dominio non degeneratezza}

\item\label{ipo dominio per dini} {for $\rho$-a.e. $x\in G^{-1}(0)$, $G$ is twice differentiable along $H$ at $x$, i.e., for $\rho$-a.e. $x\in G^{-1}(0)$ there exists $\nabla_HG(x)\in H$ and a Hilbert-Schmidt operator $\nabla_H^2G(x)$ such that
\begin{align}
\lim_{t\rightarrow0}& \frac{G(x+th)-G(x)}{t}=\langle \nabla_HG(x),h\rangle_H, \quad \textrm{uniformly with respect to $h\in H$ of norm $1$,} \label{nutella}\\
\lim_{t\rightarrow0}& \frac{\nabla_HG(x+th)-\nabla_HG(x)}{t}=\nabla_H^2G(x)h, \quad \textrm{uniformly with respect to $h\in H$ of norm $1$.} \label{bombolone}
\end{align}
Here $\rho$ is the Feyel--de La Pradelle Hausdorff--Gauss surface measure (see \cite{FP91})};

\item  $\abs{\nabla_H G(x)}_H\neq 0$ for $\rho$-a.e $x\in G^{-1}(0)$.\label{ipo dominio non degeneratezza 2}
\end{enumerate}
\end{hyp}
\noindent Hypotheses \ref{ipotesi dominio}\eqref{ipo dominio convessita}-\eqref{ipo dominio non degeneratezza} are taken from \cite{CL14} and \cite{CF16convex} in order to define traces of Sobolev functions on level sets of $G$ and to get maximal Sobolev regularity estimates for elliptic equations associated to the operator $L_{\nu,\Omega}$. In particular, Hypothesis \ref{ipotesi dominio}\eqref{ipo dominio convessita e chiusura} implies that the distance function $d_H(\cdot,\Omega)$ introduced in Section \ref{Maximal Sobolev regularity} enjoys good properties. Hypotheses \ref{ipotesi dominio}\eqref{ipo dominio per dini}-\eqref{ipo dominio non degeneratezza 2} allow us to prove Lemma \ref{Lemma derivata seconda G} which is generalization of a classical result in differential geometry (see \cite{Lan99}, \cite{BF04} and \cite{Cap16}).

\begin{hyp}\label{ipotesi peso}
$U:X\ra\R\cup\set{+\infty}$ is a proper, convex, lower semicontinuous and twice continuously differentiable along $H$ function belonging to $W^{2,t}(X,\mu)$ for some $t>3$  (see Section \ref{Notations and preliminaries} for the definition of differentiability along $H$). We set
\begin{align*}
\label{pasta}
\nu:=e^{-U}\mu.
\end{align*}
\end{hyp}
\noindent
The assumption $t>3$ may sound strange, but it is helpful to define the weighted Sobolev spaces $W^{1,2}(X,\nu)$. Indeed, let us observe that, by \cite[Lemma 7.5]{AB06}, $e^{-U}$ belongs to $W^{1,r}(X,\mu)$ for every $r<t$. Thus if $U$ satisfies Hypothesis \ref{ipotesi peso}, then it satisfies \cite[Hypothesis 1.1]{Fer15}; namely $e^{-U}\in W^{1,s}(X,\mu)$ for some $s> 1$ and $U\in W^{1,r}(X,\mu)$ for some $r>s'$.
Then following \cite{Fer15} it is possible to define the space $W^{1,2}(X,\nu)$ as the domain of the closure of the gradient operator along $H$ (see Section \ref{Notations and preliminaries} for an in-depth discussion).

From here on, we will denote by $\trace$ the trace operator acting on Sobolev functions (see Section \ref{Traces of Sobolev functions}), by $\rho$ the Feyel--de La Pradelle Hausdorff--Gauss surface measure (see \cite{Fey01}) and by $\fcon^2_b(\Omega)$ the space of the restriction to $\Omega$ of cylindrical twice differentiable functions on $X$ with bounded derivatives (see Section \ref{Special classes of functions}). We remark that, by \cite[Theorem 3.1(2)]{FU00}, $\langle\nabla^2_H U (x)h,h\rangle_H\geq 0$, for $\mu$-a.e. $x\in X$ and every $h\in H$. An important space in our investigation is
\begin{gather*}
W_U^{2,2}(\Omega,\nu)=\bigg\{u\in W^{2,2}(\Omega,\nu)\,\bigg|\,\int_\Omega\gen{\nabla_H^2 U\nabla_H u,\nabla_H u}d\nu<+\infty\bigg\},
\end{gather*}
endowed with the norm
\begin{gather}\label{norma}
\norm{u}_{W^{2,2}_U(\Omega,\nu)}^2=\norm{u}_{W^{2,2}(\Omega,\nu)}^2+\int_\Omega\gen{\nabla_H^2 U\nabla_H u,\nabla_H u}_Hd\nu.
\end{gather}
\noindent We remark that $W_U^{2,2}(\Omega,\nu)$ is a Hilbert space. We will also study the following subspace of $W^{2,2}_U(\Omega,\nu)$ 
\begin{gather*}
W_{U,N}^{2,2}(\Omega,\nu)=\bigg\{u\in W^{2,2}(\Omega,\nu)\,\bigg|\,\int_\Omega\gen{\nabla_H^2 U\nabla_H u,\nabla_H u}d\nu<+\infty,\phantom{aaaaaaaaaaaaaaaaaaaaaaa}\\
\phantom{aaaaaaaaaaaaaaaaaa}\gen{\trace(\nabla_H u),\trace(\nabla_H G)}_H=0\text{ $\rho$-a.e. in }G^{-1}(0)\bigg\}
\end{gather*}
endowed with the norm \eqref{norma}.

Our main results are the following characterizations of the  domain of the self-adjoint operator $L_{\nu,\Omega}$ 
when $\Omega$ is the whole space or a half-space. We recall that by $\norm{\cdot}_{D(L_{\nu,\Omega})}$ we denote the graph norm, i.e. for $u\in D(L_{\nu,\Omega})$
\begin{gather*}
\norm{u}_{D(L_{\nu,\Omega})}^2:=\norm{u}_{\elle^2(X,\nu)}^2+\norm{L_{\nu,\Omega} u}^2_{\elle^2(X,\nu)}.
\end{gather*}

\begin{thm}\label{cor whole space}
Assume that Hypothesis \ref{ipotesi peso} holds and that $\fcon_b^2(X)$ is dense in $W^{2,2}_U(X,\nu)$. Then $D(L_{\nu,X})= W^{2,2}_U(X,\nu)$. Moreover, for every $u\in D(L_{\nu,X})$, it holds
\begin{gather*}
\norm{u}_{D(L_{\nu,X})}\leq\norm{u}_{W_U^{2,2}(X,\nu)}\leq 2\sqrt{2}\norm{u}_{D(L_{\nu,X})}.
\end{gather*}
and fixed any orthornomal basis $\{h_n\,|\,n\in\N\}$ of $H$
\begin{gather*}
L_{\nu,X} u=\sum_{n=1}^{+\infty}\pa{\partial_{nn}u-\partial_nu\partial_nU-\partial_nu\hat{h}_n},
\end{gather*}
where the series converges in $\elle^2(X,\nu)$ (See Section \ref{Notations and preliminaries} for the definition of the $\hat{\cdot}$ operator).
\end{thm}
We remark that if the weight $U$ is such that $\nabla_H U$ is Lipschitz continuous, or more generally $H$-Lipschitz (see Section \ref{Notations and preliminaries}), then $\fcon^2_b(X)$ is dense in $W^{2,2}_U(X,\nu)$, so that the assumption of Theorem \ref{cor whole space} is satisfied (see Corollary \ref{corollario lip}). 

When $G=x^*-r$ where $x^*\in X^*\ssm\set{0}$ and $r\in\R$, i.e. if $\Omega$ is a half-space, we want to remark that the Neumann boundary condition: $\gen{\trace(\nabla_H u),\trace(\nabla_H G)}_H=0$ 
for $\rho$-a.e. $x\in G^{-1}(0)$, read 
\[x^*\pa{\trace(\nabla_H u)(x)}=\gen{\trace(\nabla_H u)(x),h_{x^*}}_H=0\] 
for $\rho$-a.e. $x\in G^{-1}(0)$, where $h_{x^*}$ is the unique vector of $H$ such that 
\begin{gather}\label{mora}
x^*(h)=\gen{h_{x^*},h}_H\text{ for every }h\in H. 
\end{gather}
Such an element exists since $x^*$ is a continuous linear functional on $H$.

\begin{thm}\label{cor halfspaces}
Assume that Hypothesis \ref{ipotesi peso} holds and $G$ is an affine function, namely $G=x^*-r$ where $x^*\in X^*\ssm\set{0}$ and $r\in\R$. If the space
\[\mathcal{Z}(\Omega)=\set{u\in\fcon_b^2(\Omega)\tc x^*\pa{\trace(\nabla_H u)(x)}=0\text{ for }\rho\text{-a.e. }x\in G^{-1}(0)},\]
where $h_{x^*}$ is defined in \eqref{mora}, is dense in the space of $W^{2,2}_{U,N}(\Omega,\nu)$, then $D(L_{\nu,\Omega})= W^{2,2}_{U,N}(\Omega,\nu)$. 
Moreover, for every $u\in D(L_{\nu,\Omega})$, it holds
\begin{gather*}
\norm{u}_{D(L_{\nu,\Omega})}\leq\norm{u}_{W_U^{2,2}(\Omega,\nu)}\leq 2\sqrt{2}\norm{u}_{D(L_{\nu,\Omega})}.
\end{gather*}
and fixed any orthornomal basis $\{h_n\,|\, n\in\N\}$ of $H$
\begin{gather*}
L_{\nu,\Omega} u=\sum_{n=1}^{+\infty}\pa{\partial_{nn}u-\partial_nu\partial_nU-\partial_nu\hat{h}_n},
\end{gather*}
where the series converges in $\elle^2(\Omega,\nu)$ (See Section \ref{Notations and preliminaries} for the definition of the $\hat{\cdot}$ operator).
\end{thm}
We remark that showing the density of $\mathcal{Z}(\Omega)$ in $W^{2,2}_{U,N}(\Omega)$ is not an easy task. This difficulty can be overcome if $\Omega$ belongs to the class
of \emph{Neumann extension domains}.
\begin{defn}
Let $Z_U^{2,2}(X,\nu)$ be the completion of the space $\fcon^2_b(X)$ with respect to the norm defined in \eqref{norma}. We say that $\Omega$ is a Neumann extension domain if there exists a linear operator $E^{\mathcal N}$ from $W^{2,2}_{U,N}(\Omega,\nu)$ into $Z_U^{2,2}(X,\nu)$ such that for every $\phi\in W^{2,2}_{U,N}(\Omega,\nu)$
\begin{enumerate}
\item $E^{\mathcal N} \phi(x)=\phi(x)$
for $\nu$-a.e $x\in \Omega$;

\item there is $K>0$, independent of $\phi$, such that $\|E^{\mathcal N}\phi\|_{Z^{2,2}_U(X,\nu)}\leq K\|\phi\|_{W^{2,2}_U(\Omega,\nu)}$.
\end{enumerate}
The operator $E^{\mathcal{N}}$ is called \emph{Neumann extension operator}.
\end{defn}

\begin{thm}\label{thm extension}
Assume that Hypothesis \ref{ipotesi peso} holds and that $\Omega$ is a Neumann extension domain satisfying Hypothesis \ref{ipotesi dominio}. Then $D(L_{\nu,\Omega})= W^{2,2}_{U,N}(\Omega,\nu)$. Moreover, for every $u\in D(L_{\nu,\Omega})$, it holds
\begin{gather*}
\norm{u}_{D(L_{\nu,\Omega})}\leq\norm{u}_{W_U^{2,2}(\Omega,\nu)}\leq 2\sqrt{2}\norm{u}_{D(L_{\nu,\Omega})}.
\end{gather*}
and fixed any orthornomal basis $\{h_n\,|\,n\in\N\}$ of $H$
\begin{gather*}
L_{\nu,\Omega} u=\sum_{n=1}^{+\infty}\pa{\partial_{nn}u-\partial_nu\partial_nU-\partial_nu\hat{h}_n},
\end{gather*}
where the series converges in $\elle^2(\Omega,\nu)$ (See Section \ref{Notations and preliminaries} for the definition of the $\hat{\cdot}$ operator).
\end{thm}
The characterization of Neumann extension domains is an open problem in Wiener space theory. The only known results are mainly negative (see \cite{BPS14}), but if $\Omega$ is a half-space and $U\equiv 0$, it is known that an extension operator can be constructed (see \cite{BPS14}). Since we were unable to find explicit computations in the literature, we made them in Lemma \ref{Extension}. Applying Theorems \ref{cor halfspaces}, \ref{thm extension} and Lemma \ref{Extension} we get the following characterization of the domain of the Ornstein--Uhlenbeck operator on half-spaces, i.e. $U\equiv 0$ and $G$ is an affine function.

\begin{thm}\label{thm halfspaces}
Assume that Hypothesis \ref{ipotesi peso} holds and $G$ is an affine function, namely $G(x)=x^*(x)-r$ with $x^*\in X^*\ssm\set{0}$ and $r\in\R$. Then
\[D(L_{\mu,\Omega})= \set{u\in W^{2,2}(\Omega,\mu)\tc x^*\pa{\trace(\nabla_H u)(x)}=0\text{ for }\rho\text{-a.e. }x\in G^{-1}(0)},\]
where $h_{x^*}$ is defined in \eqref{mora}. Moreover, for every $u\in D(L_{\mu,\Omega})$, it holds
\begin{gather*}
\norm{u}_{D(L_{\mu,\Omega})}\leq\norm{u}_{W^{2,2}(\Omega,\mu)}\leq 2\sqrt{2}\norm{u}_{D(L_{\mu,\Omega})}.
\end{gather*}
and fixed any orthornomal basis $\{h_n\,|\,n\in\N\}$ of $H$
\begin{gather*}
L_{\mu,\Omega} u=\sum_{n=1}^{+\infty}\pa{\partial_{nn}u-\partial_nu\hat{h}_n},
\end{gather*}
where the series converges in $\elle^2(\Omega,\mu)$ (See Section \ref{Notations and preliminaries} for the definition of the $\hat{\cdot}$ operator). In addition the space
\[\mathcal{Z}(\Omega)=\set{u\in\fcon_b^2(\Omega)\tc x^*\pa{\trace(\nabla_H u)(x)}=0\text{ for }\rho\text{-a.e. }x\in G^{-1}(0)}\]
is dense in $D(L_{\mu,\Omega})$ with respect to the graph norm.
\end{thm}

The paper is organized as follows: in Section \ref{Notations and preliminaries} we recall some basic definitions and we fix the notations. Section \ref{Second-order analysis of the Moreau--Yosida approximations along $H$} is dedicated to the study of the second order analysis of the Moreau--Yosida approximations along $H$, that are used to prove Theorems \ref{cor whole space}.
In section \ref{The divergence operator} we will introduce the divergence operator $\diver_{\nu,\Omega}$ as minus the formal adjoint of the gradient operator along $H$ and investigate its properties. Namely, consider the space
\begin{gather}\label{definizione Z(Omega, H)}
\mathcal{Z}(\Omega,H):=\set{\Phi:\Omega\ra H \tc\begin{array}{c}
\text{there exists $n\in\N$ and $\set{h_1,\ldots, h_n}\subseteq H$}\\
\text{such that $\Phi=\sum_{i=1}^{n}\varphi_i h_i$ for some $n\in\N$,}\\
\text{and $\varphi_i\in\fcon_b^2(\Omega)$ for $i=1,\ldots,n$.}\\
\text{In addition $\gen{\Phi,\trace(\nabla_H G)}_H=0$ $\rho$-a.e. in $G^{-1}(0)$.}
\end{array}}.
\end{gather}
For every $\Phi\in\mathcal{Z}(\Omega,H)$ put
\begin{gather}\label{norma divergenza su dominio}
\norm{\Phi}_{Z_U^{1,2}(\Omega,\nu;H)}^2:=\norm{\Phi}^2_{W^{1,2}(\Omega,\nu;H)}+\int_\Omega\gen{\nabla_H^2U\Phi,\Phi}_Hd\nu+\int_{G^{-1}(0)}\gen{\trace(\nabla^2_H G)\Phi,\Phi}_H\frac{e^{-\trace(U)}}{\abs{\trace(\nabla_H G)}_H}d\rho.
\end{gather}
Let $Z_U^{1,2}(\Omega,\nu;H)$ be the completion of the space $\mathcal{Z}(\Omega,H)$ with respect to the norm defined in \eqref{norma divergenza su dominio}. As usual the elements of $Z_U^{1,2}(\Omega,\nu;H)$ can be identified as equivalence classes of vector fields with respect to the $\nu$-a.e. equivalence relation. It is easy to see that $Z_U^{1,2}(\Omega,\nu;H)$ is a Hilbert space. In Proposition \ref{divergence for W12} we will prove that the space $Z_U^{1,2}(\Omega,\nu;H)$ is contained in the domain of the divergence operator $\diver_{\nu,\Omega}$ in $\elle^2$ and $\diver_{\nu,\Omega} \Phi\in\elle^2(\Omega,\nu)$ for every $\Phi\in Z_U^{1,2}(X,\nu;H)$. Furthermore an explicit formula for the calculation of $\diver_{\nu,\Omega}$ is given by (\ref{Formula divergenza}).

We remark that without loss of generality we can assume that the sequence $\set{h_1,\ldots,h_n}$ in \eqref{definizione Z(Omega, H)} is a sequence of orthonormal elements of $H$ (indeed, it is enough to apply the Gram-Schmidt procedure).
Moreover, we stress that the boundary integral in \eqref{norma divergenza su dominio} in general cannot be estimated by the $W^{1,2}$-norm of $\Phi$. This fact depends not only from the presence of the second order derivatives of $G$, but also from the trace theory in infinite dimensions. Indeed, as shown in \cite{Fer15} the trace of $f\in W^{1,p}(\Omega)$ belongs to $\elle^q(G^{-1}(0),e^{-U}\rho)$ for any $q\in [1,p(t-2)(t-1)^{-1}]$, where $t$ is the number fixed in Hypothesis \ref{ipotesi peso}. In particular if $p=2$ then we do not know if the trace operator is continuous in $\elle^2(G^{-1}(0), e^{-U}\rho)$.

In Section \ref{Maximal Sobolev regularity} we obtain maximal Sobolev regularity estimates for the weak solution of the problem
\begin{gather}\label{Problema 1}
\lambda u(x)-L_{\nu,\Omega}u(x)=f(x)\qquad \mu\text{-a.e. }x\in \Omega,
\end{gather}
where $\lambda >0$, and $f\in\elle^2(\Omega,\nu)$. We say that $u\in W^{1,2}(\Omega,\nu)$ is a \emph{weak solution} of problem \eqref{Problema 1} if
\[\lambda\int_\Omega u\varphi d\nu+\int_\Omega\gen{\nabla_H u,\nabla_H\varphi}_Hd\nu=\int_\Omega f\varphi d\nu\qquad\text{for every }\varphi\in W^{1,2}(\Omega,\nu).\]
Notice that the unique weak solution $u$ of problem \eqref{Problema 1} satisfies $u=R(\lambda,L_{\nu,\Omega})f$, where $R(\lambda,L_{\nu,\Omega})$ is the resolvent of $L_{\nu,\Omega}$. We recall that results about existence, uniqueness and regularity of the weak solution of problem \eqref{Problema}, in domains with sufficiently regular boundary, are known in the finite dimensional case (see the classical books \cite{GT01} and \cite{LU68} for a bounded $\Omega$ and \cite{BF04}, \cite{DPL04}, \cite{LMP05}, \cite{DPL07} and \cite{DPL08} for an unbounded $\Omega$).
If $X$ is infinite dimensional maximal Sobolev regularity results are known when $X$ is a separable Hilbert space. See for example \cite{BDPT09} and \cite{BDPT11} where $U\equiv 0$ and \cite{DPL15} where $U$ is bounded from below. When $\Omega=X$ more results are known, see for example \cite{DPG01}, \cite{MPRS02} and \cite{BL07} if $X$ is finite dimensional, \cite{DPL14} if $X$ is a Hilbert space and \cite{CF16} if $X$ is a separable Banach space. If $X$ is general separable Banach space and $\Omega\varsubsetneq X$, then the only results regarding maximal Sobolev regularity are the one contained in \cite{Cap16}, where the second named author studied problem \eqref{Problema} when $U\equiv 0$, namely when $L_{\nu,\Omega}$ is the Ornstein--Uhlenbeck operator on $\Omega$, and in \cite{CF16convex}, where the second and third named authors studied the general case.

In Section \ref{Proof of the main result and some observations} we prove Theorems \ref{cor whole space}, \ref{cor halfspaces} and \ref{thm extension} and some related corollaries.
Finally, in Section \ref{Examples} we provide some examples to which our results can be applied. In particular we study the case when $\Omega$ is the unit ball of a Hilbert space and we prove Theorem \ref{thm halfspaces}.

\section{Notation and preliminaries} \label{Notations and preliminaries}

We will denote by $X^*$ the topological dual of $X$. We recall that $X^*\subseteq\elle^2(X,\mu)$. The linear operator $R_\mu:X^*\ra (X^*)'$
\begin{gather}\label{operatore di covariaza}
R_\mu x^*(y^*)=\int_X x^*(x)y^*(x)d\mu(x)
\end{gather}
is called the covariance operator of $\mu$. Since $X$ is separable, then it is actually possible to prove that $R_\mu:X^*\ra X$ (see \cite[Theorem 3.2.3]{Bog98}). We denote by $X^*_\mu$ the closure of $X^*$ in $\elle^2(X,\mu)$. The covariance operator $R_\mu$ can be extended by continuity to the space $X^*_\mu$, still by formula \eqref{operatore di covariaza}. By \cite[Lemma 2.4.1]{Bog98} for every $h\in H$ there exists a unique $g\in X^*_\mu$ with $h= R_\mu g$, in this case we set
\begin{gather}\label{definizione hat}
\hat{h}:=g.
\end{gather}

Throughout the paper we fix an orthonormal basis $\set{e_i}_{i\in\N}$ of $H$ such that $\hat{e}_i$ belongs to $X^*$, for every $i\in\N$. Such basis exists by \cite[Corollary 3.2.8(ii)]{Bog98}.

\subsection{Differentiability along $H$}

We say that a function $f:X\ra\R$ is \emph{differentiable along $H$ at $x$} if there exists $v\in H$ such that
\[\lim_{t\ra 0}\frac{f(x+th)-f(x)}{t}=\gen{v,h}_H,\]
uniformly with respect to $h\in H$, with $\abs{h}_H=1$. In this case, the vector $v\in H$ is unique and we set $\nabla_H f(x):=v$. Moreover, for every $k\in\N$ the derivative of $f$ in the direction of $e_k$ exists and it is given by
\begin{gather*}
\partial_k f(x):=\lim_{t\ra 0}\frac{f(x+te_k)-f(x)}{t}=\gen{\nabla_H f(x),e_k}_H.
\end{gather*}

We denote by $\mathcal{H}_2$ the space of the Hilbert--Schmidt operators in $H$, that is the space of the bounded linear operators $A:H\ra H$ such that $\norm{A}_{\mathcal{H}_2}^2=\sum_{i}\abs{Ae_i}^2_H$ is finite (see \cite{DU77}).
We say that a function $f:X\ra\R$ is \emph{twice differentiable along $H$ at $x$} if it is differentiable along $H$ at $x$ and there exists $A\in\mathcal{H}_2$ such that
\[H\text{-}\lim_{t\ra 0}\frac{\nabla_Hf(x+th)-\nabla_Hf(x)}{t}=A h,\]
uniformly with respect to $h\in H$, with $\abs{h}_H=1$. In this case the operator $A$ is unique and we set $\nabla_H^2 f(x):=A$. Moreover, for every $i,j\in\N$ we set
\begin{gather*}
\partial_{ij} f(x):=\lim_{t\ra 0}\frac{\partial_jf(x+te_i)-\partial_jf(x)}{t}=\langle\nabla_H^2 f(x)e_j,e_i\rangle_H.
\end{gather*}

\subsection{Special classes of functions}\label{Special classes of functions}

For $k\in\N\cup\set{\infty}$, we denote by $\fcon^k(X)$ ($\fcon_b^k(X)$ respectively) the space of the cylindrical function of the type
\(f(x)=\varphi(x^*_1(x),\ldots,x^*_n(x))\)
where $\varphi\in\con^{k}(\R^n)$ ($\varphi\in\con^{k}_b(\R^n)$, respectively) and $x^*_1,\ldots,x^*_n\in X^*$, for some $n\in\N$. We remark that $\fcon^\infty_b(X)$ is dense in $\elle^p(X,\nu)$ for all $p\geq 1$ (see \cite[Proposition 3.6]{Fer15}). We recall that if $f\in \fcon^2(X)$, then $\partial_{ij}f(x)=\partial_{ji}f(x)$ for every $i,j\in\N$ and $x\in X$.

If $Y$ is a Banach space, a function $F:X\ra Y$ is said to be $H$-Lipschitz if there exists a positive constant $C$ such that
\begin{gather}\label{costante di H-lip}
\norm{F(x+h)-F(x)}_Y\leq C\abs{h}_H,
\end{gather}
for every $h\in H$ and $\mu$-a.e. $x\in X$ (see \cite[Section 4.5 and Section 5.11]{Bog98}). We denote with $[F]_{H\text{-Lip}}$ the best constant $C$ appearing in \eqref{costante di H-lip}.

A function $F:X\ra\R$ is said to be $H$-continuous, if \(\lim_{ \abs{h}_H\ra 0}F(x+h)=F(x)\), for $\mu$-a.e. $x\in X$.

\subsection{Sobolev spaces}\label{sobolev_spaces}
The Gaussian Sobolev spaces $W^{1,p}(X,\mu)$ and $W^{2,p}(X,\mu)$, with $p\geq 1$, are the completions of the smooth cylindrical functions $\fcon_b^\infty(X)$ in the norms
\begin{gather*}
\norm{f}_{W^{1,p}(X,\mu)}:=\norm{f}_{\elle^p(X,\mu)}+\pa{\int_X\abs{\nabla_H f(x)}_H^pd\mu(x)}^{\frac{1}{p}};\\
\norm{f}_{W^{2,p}(X,\mu)}:=\norm{f}_{W^{1,p}(X,\mu)}+\pa{\int_X\norm{\nabla_H^2 f(x)}^p_{\mathcal{H}_2}d\mu(x)}^{\frac{1}{p}}.
\end{gather*}
Such spaces can be identified with subspaces of $\elle^p(X,\mu)$ and the (generalized) gradient and Hessian along $H$, $\nabla_H f$ and $\nabla_H^2 f$, are well defined and belong to $\elle^p(X,\mu;H)$ and $\elle^p(X,\mu;\mathcal{H}_2)$, respectively. The spaces $W^{1,p}(X,\mu;H)$ are defined in a similar way, replacing smooth cylindrical functions with $H$-valued smooth cylindrical functions (i.e. the linear span of the functions $x\mapsto f(x)h$, where $f$ is a smooth cylindrical function and $h\in H$). For more information see \cite[Section 5.2]{Bog98}.

Now we consider $\nabla_H:\fcon^\infty_b(X)\ra\elle^p(X,\nu;H)$. This operator is closable in $\elle^p(X,\nu)$ whenever $p>\frac{t-1}{t-2}$ (see \cite[Definition 4.3]{Fer15}). For such $p$ we denote by $W^{1,p}(X,\nu)$ the domain of its closure in $\elle^p(X,\nu)$. In the same way the operator $(\nabla_H,\nabla^2_H):\fcon^\infty_b(X)\ra \elle^p(X,\nu;H)\times\elle^p(X,\nu;\mathcal{H}_2)$ is closable in $\elle^p(X,\nu)$, whenever $p>\frac{t-1}{t-2}$ (see \cite[Proposition 2.1]{CF16}). For such $p$ we denote by $W^{2,p}(X,\nu)$ the domain of its closure in $\elle^p(X,\nu)$. The spaces $W^{1,p}(X,\nu;H)$ are defined in a similar way, replacing smooth cylindrical functions with $H$-valued smooth cylindrical functions.

We want to point out that if Hypothesis \ref{ipotesi peso} holds, then $\frac{t-1}{t-2}<2$. In particular the above arguments allows us to define the Sobolev spaces $W^{1,2}(X,\nu)$ and $W^{2,2}(X,\nu)$.

We shall use the integration by parts formula (see \cite[Lemma 4.1]{Fer15}) for $\varphi\in W^{1,p}(X,\nu)$ with $p>\frac{t-1}{t-2}$:
\begin{gather*}
\int_X\partial_k\varphi d\nu=\int_X\varphi(\partial_kU+\hat{e}_k)d\nu\qquad\text{ for every }k\in\N,
\end{gather*}
where $\hat{e}_k$ is defined in formula \eqref{definizione hat}.
Finally, we recall that if $U$ satisfies Hypothesis \ref{ipotesi peso} then for every $u\in\fcon^2_b(X)$
\begin{gather}\label{formula Lnu}
L_{\nu,X} u=\sum_{i=1}^{+\infty}\pa{\partial_{ii}u-\pa{\partial_i U+\hat{e}_i}\partial_iu},
\end{gather}
where the series converges in $\elle^2(X,\nu)$ (see \cite[Proposition 5.3]{Fer15}).

\subsection{Capacity}

Let $L_p$ be the infinitesimal generator of the \emph{Ornstein--Uhlenbeck semigroup} $T(t)$ in $\elle^p(X,\mu)$, where
\[T(t)f(x):=\int_Xf\pa{e^{-t}x+(1-e^{-2t})^{\frac{1}{2}}y}d\mu(y)\qquad\text{ for }t>0.\]
For $k=1,2,3$, we define the \emph{$C_{k,p}$-capacity} of an open set $A\subseteq X$ as
\[C_{k,p}(A):=\inf\set{\norm{f}_{\elle^p(X,\mu)}\tc (I-L_p)^{-\frac{k}{2}}f\geq 1\ \mu\text{-a.e. in }A}.\]
For a general Borel set $B\subseteq X$ we let $C_{k,p}(B)=\inf\set{C_{k,p}(A)\tc B\subseteq A\text{ open}}$. By $f\in W^{k,p}(X,\mu)$ we mean an equivalence class of functions and we call every element ``version''. For any $f\in W^{k,p}(X,\mu)$ there exists a version $\ol f$ of $f$ which is Borel measurable and \emph{$C_{k,p}$-quasicontinuous}, i.e. for every $\eps>0$ there exists an open set $A\subseteq X$ such that $C_{k,p}(A)\leq \eps$ and $\ol{f}_{|_{X\ssm A}}$ is continuous. Furthermore, for every $r>0$
\[C_{k,p}\pa{\set{x\in X\tc \abs{\ol{f}(x)}>r}}\leq\frac{1}{r}\norm{(I-L_p)^{-\frac{k}{2}}\ol{f}}_{\elle^p(X,\mu)}.\]
See \cite[Theorem 5.9.6]{Bog98}. Such a version is called a \emph{$(k,p)$-precise version of $f$}. Two precise versions of the same $f$ coincide outside sets with null $C_{k,p}$-capacity. All our results will be independent on our choice of a precise version of $G$ in Hypothesis \ref{ipotesi dominio}. With obvious modification the same definition can be adapted to functions belonging to $W^{k,p}(X,\mu;H)$ and $W^{k,p}(X,\mu;\mathcal{H}_2)$.

\subsection{Sobolev spaces on sublevel sets}

The proof of the results stated in this subsection can be found in \cite{CL14} and \cite{Fer15}. Let $G$ be a function satisfying Hypothesis \ref{ipotesi dominio}. We are interested in Sobolev spaces on sublevel sets of $G$.

For $k\in \N\cup\set{\infty}$, we denote by $\fcon^k_b(\Omega)$ the space of the restriction to $\Omega$ of functions in $\fcon^k_b(X)$. For any $p\geq1$, the spaces $W^{1,p}(\Omega,\mu)$ and $W^{2,p}(\Omega,\mu)$ are defined as the domain of the closure of the operators $\nabla_H:\fcon_b^\infty(\Omega)\ra \elle^p(\Omega,\mu;H)$ and $(\nabla_H,\nabla_H^2):\fcon^\infty_b(\Omega)\ra\elle^p(\Omega,\mu;H)\times\elle^p(\Omega,\mu;\mathcal{H}_2)$, respectively. See \cite[Lemma 2.2]{CL14} and \cite[Proposition 1]{Cap16}.

We recall that $\nabla_H:\fcon^\infty_b(\Omega)\ra \elle^p(\Omega,\nu; H)$ and $(\nabla_H,\nabla^2_H):\fcon_b^\infty(\Omega)\ra \elle^p(\Omega,\nu;H)\times\elle^p(\Omega,\nu;\mathcal{H}_2)$ are closable operators in $\elle^p(\Omega,\nu)$, whenever $p>\frac{t-1}{t-2}$ (see \cite[Proposition 6.1]{Fer15} and \cite[Proposition 2.2]{CF16convex}). For such values of $p$ we denote by $W^{1,p}(\Omega,\nu)$ the domain of its closure in $\elle^p(\Omega,\nu)$ and we will still denote by $\nabla_H$ the closure operator. The space $W^{2,p}(\Omega,\nu)$ is defined in the same way.

Finally we want to remark that if Hypotheses \ref{ipotesi dominio} and \ref{ipotesi peso} hold, then $\frac{t-1}{t-2}<2$. In particular the Sobolev spaces $W^{1,2}(\Omega,\nu)$ and $W^{2,2}(\Omega,\nu)$ are well defined.

\subsection{Traces of Sobolev functions}\label{Traces of Sobolev functions}
By $\rho$ we indicate the Feyel--de La Pradelle Hausdorff--Gauss surface measure. For a comprehensive treatment of surface measures in infinite dimensional Banach spaces with Gaussian measures we refer to \cite{FP91}, \cite{Fey01} and \cite{CL14}.

Traces of Sobolev functions in infinite dimensional Banach spaces have been studied in \cite{CL14} in the Gaussian case and in \cite{Fer15} in the weighted Gaussian case. We stress that in \cite{CL14} the definition of Sobolev Spaces is different with respect to the our one, but these two definitions coincide in the case of Gaussian measure.
Assume that Hypotheses \ref{ipotesi dominio} and \ref{ipotesi peso} hold and let $p>\frac{t-1}{t-2}$. If $\varphi\in W^{1,p}(\Omega,\nu)$ we define the trace of $\varphi$ on $G^{-1}(0)$ as follows:
\[\trace\varphi=\lim_{n\ra+\infty}\varphi_{n_{|_{G^{-1}(0)}}}\qquad\text{in }\elle^{1}(G^{-1}(0),e^{-U}\rho),\]
and it is possible to prove that ${\rm Tr}\varphi\in \elle^{q}(G^{-1}(0),e^{-U}\rho)$ for any $q\in[1,p(t-2)(t-1)^{-1}]$, where $t$ is the real number fixed in Hypothesis \ref{ipotesi peso}.
Here, $(\varphi_{n})_{n\in\N}$ is any sequence in $\lip_b(\Omega)$, the space of bounded and Lipschitz functions on $\Omega$, which converges in $W^{1,p}(\Omega,\nu)$ to $\varphi$. The definition does not depend on the choice of the sequence $(\varphi_n)_{n\in\N}$ in $\lip_b(\Omega)$ approximating $\varphi$ in $W^{1,p}(\Omega,\nu)$ (see \cite[Proposition 7.1]{Fer15}). In addition the following result holds.

\begin{pro}\label{trace continuity}
Assume that Hypotheses \ref{ipotesi dominio} and \ref{ipotesi peso} hold. Then the operator \(\trace:W^{1,p}(\Omega,\nu)\ra\elle^q(G^{-1}(0),e^{-U}\rho)\) is continuous for every $p>\frac{t-1}{t-2}$ and $q\in \sq{1,p\frac{t-2}{t-1}}$. Moreover, if $U\equiv 0$, then the trace operator is continuous from $W^{1,p}(\Omega,\mu)$ to $\elle^q(G^{-1}(0),\rho)$ for every $p>1$ and $q\in[1,p)$ (see \cite[Corollary 4.2]{CL14} and \cite[Corollary 7.3]{Fer15}).
\end{pro}

We will still denote by $\trace\Psi=\sum_{n=1}^{+\infty}(\trace\psi_n)e_n$ if $\Psi\in W^{1,p}(\Omega,\nu;H)$, for $p>\frac{t-1}{t-2}$, and $\psi_n=\gen{\Psi,e_n}_H$. The main result of \cite{Fer15} is the following integration by parts formula.

\begin{thm}\label{divergence theorem with traces}
Assume that Hypotheses \ref{ipotesi dominio} and \ref{ipotesi peso} hold and let $p>\frac{t-1}{t-2}$. For every $\varphi\in W^{1,p}(\Omega,\nu)$ and $k\in\N$ we have
\[\int_{\Omega}\pa{\partial_k\varphi-\varphi\partial_k U-\varphi\hat{e}_k}d\nu=\int_{G^{-1}(0)}\trace(\varphi) \trace\pa{\frac{\partial_k G}{\abs{\nabla_H G}_H}}e^{-\trace(U)}d\rho.\]
\end{thm}

Another important result, that we will use in this paper, is the following (see \cite[Proposition 4.8]{CL14} and \cite[Proposition 7.5]{Fer15}).

\begin{pro}\label{coincidenza traccia e versione precisa}
Assume that Hypotheses \ref{ipotesi dominio} and \ref{ipotesi peso} hold and let $p>\frac{t-1}{t-2}$. Then for every $\varphi\in W^{1,p}(\Omega, \nu)$, the trace of $\trace(\varphi)$ at $G^{-1}(0)$ coincides $\rho$-a.e. with the restriction to $G^{-1}(0)$ of any precise version $\tilde{\varphi}$ of $\varphi$.
\end{pro}

\subsection{The spaces $W^{2,2}_{U,N}$ and $Z^{2,2}_{U,N}$}

We recall the definition of the space $W^{2,2}_U$ and $W^{2,2}_{U,N}$.
\begin{gather*}
W_U^{2,2}(\Omega,\nu)=\bigg\{u\in W^{2,2}(\Omega,\nu)\,\bigg|\,\int_\Omega\gen{\nabla_H^2 U\nabla_H u,\nabla_H u}d\nu<+\infty\bigg\},
\end{gather*}
endowed with the norm
\begin{gather}\label{norma1}
\norm{u}_{W^{2,2}_U(\Omega,\nu)}^2=\norm{u}_{W^{2,2}(\Omega,\nu)}^2+\int_\Omega\gen{\nabla_H^2 U\nabla_H u,\nabla_H u}_Hd\nu.
\end{gather}
\noindent We consider the space $W^{2,2}_U(\Omega,\nu)$ 
\begin{gather*}
W_{U,N}^{2,2}(\Omega,\nu)=\bigg\{u\in W^{2,2}(\Omega,\nu)\,\bigg|\,\int_\Omega\gen{\nabla_H^2 U\nabla_H u,\nabla_H u}d\nu<+\infty,\phantom{aaaaaaaaaaaaaaaaaaaaaaa}\\
\phantom{aaaaaaaaaaaaaaaaaa}\gen{\trace(\nabla_H u),\trace(\nabla_H G)}_H=0\text{ $\rho$-a.e. in }G^{-1}(0)\bigg\}
\end{gather*}
endowed with the norm \eqref{norma1}.

We denote by $Z_U^{2,2}(X,\nu)$ be the completion of the space $\fcon^2_b(X)$ with respect to the norm defined in \eqref{norma1} and by $Z^{2,2}_{U,N}(\Omega,\nu)$ the completion of the space
\[\mathcal{Z}(\Omega):=\set{f\in \fcon^2_b(\Omega)\tc \gen{\nabla_H f(x),\nabla_H G(x)}_H=0\text{ for $\rho$-a.e. }x\in G^{-1}(0)},\]
with respect to the norm \eqref{norma1}.

\section{Second-order analysis of the Moreau--Yosida approximations along $H$}\label{Second-order analysis of the Moreau--Yosida approximations along $H$}

We start this section by recalling the definition of the subdifferential of a convex semicontinuous function. If $f:X\ra\R$ is a proper, convex and lower semicontinuous function, we denote by $\dom(f)$ the domain of $f$, namely $\dom(f):=\set{x\in X\tc f(x)<+\infty}$,
and by $\partial f(x)$ the subdifferential of $f$ at the point $x$, i.e.
\begin{gather*}
\partial f(x):=\left\{\begin{array}{lr}
\set{x^*\in X^*\tc f(y)\geq f(x)+x^*(y-x)\text{ for every }y\in X} & x\in\dom(f);\\
\emptyset & x\notin\dom(f).
\end{array}\right.
\end{gather*}
For a classical treatment of subdifferentials of convex functions we refer to \cite{Phe93} and \cite{BP12}.

We recall that for $\alpha>0$ the \emph{Moreau--Yosida approximation along $H$} of a proper convex and lower semicontinuous function $f:X\ra\R\cup\set{+\infty}$ is
\begin{gather}\label{M env}
f_\alpha(x):=\inf\set{f(x+h)+\frac{1}{2\alpha}\abs{h}^2_H\tc h\in H}.
\end{gather}
See \cite[Section 3]{CF16} and \cite[Section 4]{CF16convex} for more details and \cite{Bre73} and \cite[Section 12.4]{BC11} for a treatment of the classical Moreau--Yosida approximations in Hilbert spaces, which are different from the ones defined in \eqref{M env}. Second-order analysis of the classical Moreau--Yosida approximations have been studied in various papers, e.g. \cite{Qi99}, \cite{OG07} and \cite{Ovc10}.

In the following proposition we recall some results contained in \cite[Section 3]{CF16} and in \cite[Section 4]{CF16convex}.

\begin{pro}\label{proprieta MY}
Let $x\in X$, $\alpha>0$ and $f:X\ra\R\cup\set{+\infty}$ be a proper convex and lower semicontinuous function. The following properties hold:
\begin{enumerate}
\item the function $g_{\alpha,x}:H\ra\R$ defined as \(g_{\alpha,x}(h):=f(x+h)+\frac{1}{2\alpha}\abs{h}^2_H\), has a unique global minimum point $P(x,\alpha)\in H$. Moreover $P(x,\alpha)\ra 0$ in $H$ as $\alpha$ goes to zero;\label{esistenza minimo}

\item $f_\alpha(x)\nearrow f(x)$ as $\alpha\ra 0^+$. In particular $f_\alpha(x)\leq f(x)$ for every $\alpha>0$ and $x\in X$;\label{convergenza MY}

\item for $p\in H$, we have $p=P(x,\alpha)$ if, and only if, \(f(x+p)\leq f(x+h)+\frac{1}{\alpha}\gen{p,h-p}_H\), for every $h\in H$;\label{caratterizzazione punto minimo MY}

\item the function $P_{x,\alpha}:H\ra H$ defined as $P_{x,\alpha}(h):=P(x+h,\alpha)$ is Lipschitz continuous, with Lipschitz constant less than or equal to $1$;

\item $f_\alpha$ is differentiable along $H$ at every point $x\in X$. In addition, for every $x\in X$, we have \(\nabla_H f_\alpha(x)=-\alpha^{-1}P(x,\alpha)\);\label{differenziabilita MY}

\item $f_\alpha$ belongs to $W^{2,p}(X,\mu)$, whenever $f\in\elle^p(X,\mu)$ for some $1\leq p<+\infty$;\label{Sobolev MY}

\item let $x\in \dom(f)$ and assume that $f$ belongs to $W^{1,p}(X,\mu)$ for some $p>1$. If we define $F:H\ra \R$ as $F(h):=f(x+h)$, then $F$ is proper convex and lower semicontinuous function. Moreover, $\nabla_H f(x)\in\partial F(0)$ and $\nabla_H f_\alpha(x)\in \partial F(P(x,\alpha))$;\label{gradiente MY}

\item let $x\in \dom(f)$ and assume that $f$ belongs to $W^{1,p}(X,\mu)$ for some $p>1$. Then $\nabla_H f_\alpha(x)$ converges to $\nabla_H f(x)$ as $\alpha$ goes to zero.\label{convergenza gradiente MY}
\end{enumerate}
\end{pro}

The last property we need is the convergence of the second-order derivative along $H$.

\begin{pro}\label{second order MY}
Let $f\in W^{2,p}(X,\mu)$ for some $p>1$ and $\alpha>0$. Assume that $f$ is twice differentiable along $H$ at every point $x\in\dom f$. Then for every $x\in dom(f)$ there exists $\nabla_H^2 f_\alpha(x)$, and $\nabla_H^2 f_\alpha(x)$ converges to $\nabla_H^2 f(x)$ as $\alpha$ goes to zero.
\end{pro}

\begin{proof}
By Proposition \ref{proprieta MY}\eqref{gradiente MY} we get $\nabla_H f_\alpha(x)=\nabla_H f(x+P(x,\alpha))$. We can differentiate along $H$ since $P(x,\alpha)$ admits a $H$-gradient (it is $H$-Lipschitz).
\begin{gather*}
\nabla_H^2f_\alpha(x)=\nabla_H^2f(x+P(x,\alpha))(I_H+\nabla_HP(x,\alpha))=\nabla_H^2f(x+P(x,\alpha))(I_H-\alpha\nabla_H^2f_\alpha(x)).
\end{gather*}
If we let $\alpha\ra0$ then, by \ref{proprieta MY}\eqref{convergenza gradiente MY}, we get $\lim_{\alpha\ra 0}\nabla_H^2f_\alpha(x)=\nabla_H^2f(x)$.
\end{proof}

\section{The divergence operator}\label{The divergence operator}

We start this section by recalling the definition of divergence, see \cite[Section 5.8]{Bog98} for the case $\Omega=X$. For every measurable map $\Phi:\Omega\ra X$ and for every $f\in\fcon^\infty_b(\Omega)$ we define
\begin{gather}\label{divergenza su dominio}
\partial_\Phi f(x)=\lim_{t\ra 0}\frac{f(x+t\Phi(x))-f(x)}{t},\qquad\qquad x\in\Omega.
\end{gather}

\begin{defn}
Let $\Phi\in \elle^1(\Omega,\nu;X)$ be a vector field. We say that $\Phi$ admits \emph{divergence} if there exists a function $g\in\elle^1(\Omega,\nu)$ such that
\begin{gather}\label{condizione di divergenza in Omega}
\int_\Omega\partial_{\Phi}fd\nu=-\int_\Omega fgd\nu,
\end{gather}
for every $f\in\fcon_b^\infty(\Omega)$, where $\partial_{\Phi}f$ has been defined in \eqref{divergenza su dominio}. If such a function $g$ exists, then we set $\diver_{\nu,\Omega} \Phi:=g$. Observe that, when $\diver_{\nu,\Omega} \Phi$ exists, it is unique by the density of $\fcon_b^\infty(\Omega)$ in $\elle^p(\Omega,\nu)$ (see \cite{Fer15}). We denote by $D(\diver_{\nu,\Omega})$ the domain of $\diver_{\nu,\Omega}$ in $\elle^1(\Omega,\nu;X)$. Lastly, we observe that if $\Phi\in\elle^1(\Omega,\nu;H)$, then $\partial_{\Phi}f(x)=\gen{\nabla_Hf(x),\Phi(x)}_H$ for $x\in\Omega$. In this case (\ref{condizione di divergenza in Omega}) becomes
\begin{gather}\label{condizione di divergenza in H}
\int_\Omega\gen{\nabla_Hf,\Phi}_Hd\nu=-\int_\Omega fgd\nu,\qquad \text{for every } f\in\fcon_b^\infty(\Omega).
\end{gather}
\end{defn}
\noindent We remark that in $\elle^2$-setting, the divergence operator $\diver_{\nu,\Omega}$ is $-\nabla_H^*$, the $\elle^2$-adjoint of the the gradient along $H$ operator. Indeed, for any $\Phi\in \elle^2(\Omega,\nu;H)$ and any $f\in W^{1,2}(\Omega,\nu)$ we get
\begin{align*}
\int_{\Omega}\langle \nabla_Hf,\Phi\rangle_Hd\nu=-\int_\Omega f\diver_{\nu,\Omega}\Phi d\nu.
\end{align*}

The following two technical lemmata are crucial to show Theorems \ref{cor whole space} and \ref{cor halfspaces}. In particular, the second one is a generalization of a well known result in differential geometry, see \cite{Lan99}, \cite{BF04} and \cite{Cap16}.

\begin{lemma}\label{Conti per divergenza con Neumann}
If Hypothesis \ref{ipotesi peso} holds, then
\begin{gather}\label{conti 2}
\int_X\pa{\partial_h f-f\partial_h U-f\hat{h}}\pa{\partial_k g-g\partial_k U-g\hat{k}}d\nu=\int_X fg\partial_h\partial_k Ud\nu+\gen{h,k}_H\int_X fgd\nu+\int_X \partial_k f\partial_h gd\nu.
\end{gather}
If $\Omega\subsetneq X$, let Hypotheses \ref{ipotesi dominio} and \ref{ipotesi peso} hold true, and let $f,g\in\fcon^2_b(\Omega)$ and $h,k\in H$. Then
\begin{gather}
\notag \int_\Omega\pa{\partial_h f-f\partial_h U-f\hat{h}}\pa{\partial_k g-g\partial_k U-g\hat{k}}d\nu=\\
\label{conti 1}=\int_{G^{-1}(0)}f\pa{\partial_k g-g\trace(\partial_k U)-g\hat{k}}\frac{\trace(\partial_h G)}{\abs{\trace(\nabla_H G)}_H}e^{-\trace(U)}d\rho-\int_{G^{-1}(0)}f\partial_h g\frac{\trace(\partial_k G)}{\abs{\trace(\nabla_H G)}_H}e^{-\trace(U)}d\rho+\\
\notag +\int_\Omega fg\partial_h\partial_k Ud\nu+\gen{h,k}_H\int_\Omega fgd\nu+\int_\Omega \partial_k f\partial_h gd\nu.
\end{gather}
\end{lemma}

\begin{proof}
We will only prove \eqref{conti 1}, since the proof of \eqref{conti 2} is essentially the same. We will use Theorem \ref{divergence theorem with traces} several times. We have
\begin{gather*}
\int_\Omega\pa{\partial_h f-f\partial_h U-f\hat{h}}\pa{\partial_k g-g\partial_k U-g\hat{k}}d\nu=\displaybreak[0]\\
=\int_\Omega\partial_h f\pa{\partial_k g-g\partial_k U-g\hat{k}}d\nu-\int_\Omega \pa{f\partial_h U+f\hat{h}}\pa{\partial_k g-g\partial_k U-g\hat{k}}d\nu=\displaybreak[0]\\
=\int_\Omega\partial_h \pa{f\pa{\partial_k g-g\partial_k U-g\hat{k}}}d\nu-\int_\Omega f\partial_h\pa{\partial_k g-g\partial_k U-g\hat{k}}d\nu+\displaybreak[0]\\
-\int_\Omega \pa{f\partial_h U+f\hat{h}}\pa{\partial_k g-g\partial_k U-g\hat{k}}d\nu=
\displaybreak[0]\\
=\int_\Omega\partial_h \pa{f\pa{\partial_k g-g\partial_k U-g\hat{k}}}-f\pa{\partial_k g-g\partial_k U-g\hat{k}}\partial_h U-f\pa{\partial_k g-g\partial_k U-g\hat{k}}\hat{h}d\nu+\displaybreak[0]\\
-\int_\Omega f\partial_h\pa{\partial_k g-g\partial_k U-g\hat{k}}d\nu=\displaybreak[0]\\
=\int_{G^{-1}(0)}f\pa{\partial_k g-g\trace(\partial_k U)-g\hat{k}}\frac{\trace(\partial_h G)}{\abs{\trace(\nabla_H G)}_H}e^{-\trace(U)}d\rho-\int_\Omega f\partial_h\pa{\partial_k g-g\partial_k U-g\hat{k}}d\nu=\displaybreak[0]\\
=\int_{G^{-1}(0)}f\pa{\partial_k g-g\trace(\partial_k U)-g\hat{k}}\frac{\trace(\partial_h G)}{\abs{\trace(\nabla_H G)}_H}e^{-\trace(U)}d\rho-\int_\Omega f\partial_k\partial_h gd\nu+\int_\Omega f\partial_h g\partial_k Ud\nu+
\displaybreak[0]\\
+\int_\Omega fg\partial_h\partial_k Ud\nu+\int_\Omega f\partial_h g\hat{k}d\nu+\gen{h,k}_H\int_\Omega fgd\nu=\displaybreak[0]\\
=\int_{G^{-1}(0)}f\pa{\partial_k g-g\trace(\partial_k U)-g\hat{k}}\frac{\trace(\partial_h G)}{\abs{\trace(\nabla_H G)}_H}e^{-\trace(U)}d\rho+\int_\Omega fg\partial_h\partial_k Ud\nu+\gen{h,k}_H\int_\Omega fgd\nu+
\displaybreak[0]\\
+\int_\Omega \partial_k f\partial_h gd\nu-\int_\Omega \partial_k\pa{f\partial_h g}-f\partial_h g\partial_k U-f\partial_h g\hat{k}d\nu=\displaybreak[0]\\
=\int_{G^{-1}(0)}f\pa{\partial_k g-g\trace(\partial_k U)-g\hat{k}}\frac{\trace(\partial_h G)}{\abs{\trace(\nabla_H G)}_H}e^{-\trace(U)}d\rho-\int_{G^{-1}(0)}f\partial_h g\frac{\trace(\partial_k G)}{\abs{\trace(\nabla_H G)}_H}e^{-\trace(U)}d\rho+
\displaybreak[0]\\
+\int_\Omega fg\partial_h\partial_k Ud\nu+\gen{h,k}_H\int_\Omega fgd\nu+\int_\Omega \partial_k f\partial_h gd\nu.
\end{gather*}
\end{proof}


\begin{lemma}\label{Lemma derivata seconda G}
Assume Hypotheses \ref{ipotesi dominio}. Let $\Phi\in \mathcal{Z}(\Omega,H)$ the space defined in \eqref{definizione Z(Omega, H)}. Then for $\rho$-a.e. $x\in G^{-1}(0)$
\begin{gather}\label{roba inutilissima}
\gen{\trace(\nabla_H^2G)(x)\Phi(x),\Phi(x)}_H=-\gen{(\nabla_H\Phi(x))\Phi(x),\trace(\nabla_H G)(x)}_H.
\end{gather}
\end{lemma}

\begin{proof}
The proof is rather long and it will be split into various steps. Let $\{h_i\}_{i\in\N}$ be the orthonormal basis of $H$ associated with $\Phi$ given by the definition of the space $\mathcal{Z}(\Omega,H)$.
By Hypothesis \ref{ipotesi dominio}, Proposition \ref{coincidenza traccia e versione precisa} and the very definition of $\mathcal{Z}(\Omega,H)$ the set
\begin{gather}\label{boh}
A=\set{x\in G^{-1}(0)\tc\begin{array}{c}
\text{$G$ is continuous and twice differentiable along $H$ at $x$,}\\
\abs{\nabla_H G(x)}_H\neq 0,\ G(x)=\trace G(x),\ \nabla_H G(x)=\trace(\nabla_H G)(x),\\
\nabla^2_H G(x)=\trace(\nabla^2_H G)(x),\ \gen{\Phi(x),\trace(\nabla_H G)(x)}_H=0.
\end{array}}
\end{gather}
has full $\rho$ measure. We will prove that \eqref{roba inutilissima} holds for every point $x_0$ belonging to $A$. By \eqref{boh} we have $\nabla_H G(x_0)\neq 0$, so there exists $n_0\in\N$ such that
\[\partial_{n_0}G(x_0)\neq 0.\]
Without loss of generality, we can assume that $n_0=1$. By the very definition of the space $\mathcal{Z}(\Omega,H)$ there exist $K(\Phi)>0$, $k\in\N$ and $(\varphi_i)_{i=1}^k\subseteq \fcon^2_b(\Omega)$ such that for every $k_1,k_2\in H$ it holds
\begin{align}
\abs{\Phi(x_0+k_1)-\Phi(x_0+k_2)}_H\leq K(\Phi)\abs{k_1-k_2}_H,
\label{cioccolato}
\end{align}
and $\Phi(x)=\sum_{i=1}^{k}\varphi_i(x)h_i$. For $i>k$ we set $\varphi_i(x)\equiv 0$.
\begin{enumerate}
\item[\textbf{Step 1:}] Let us consider the space
\[h_1^\perp=\set{\ol{h}\in H\tc \gen{\ol{h},h_1}_H=0},\]
endowed with the Hilbert space norm $\abs{\ol{h}}_{h_1^\perp}=\sum_{i=2}^{+\infty}\gen{\ol{h},h_i}_{H}^2$. We denote its inner product by $\gen{\cdot,\cdot}_{h_1^\perp}$ and recall that $\{h_i\}_{i\geq 2}$ is an orthonormal basis for $h_1^\perp$ and $H=h_1^\perp\oplus\linspan\set{h_1}$. We want to apply the implicit function theorem to a function defined on $h_1^\perp\oplus\R$. Let $G_{x_0}:h_1^\perp\oplus\R\ra\R$ be the function defined as
\[G_{x_0}((\ol{h},\alpha)):=G(x_0+\ol{h}+\alpha h_1).\]
Observe that $G_{x_0}((0,0))=G(x_0)=0$ and
\[D_2G_{x_0}((0,0))=\partial_1 G(x_0)\neq 0,\]
where $D_2$ is the derivative with respect the second variable. Since \eqref{nutella} implies that $G_{x_0}$ is Fr\'echet differentiable at $0$, applying the implicit function theorem, see \cite[Theorem 5.9]{Lan99}, we get an open neighborhood $U_0\subseteq h_1^\perp$ of the origin and a continuously Fr\'echet differentiable function $g_{x_0}: U_0\ra \R$ such that for every $\ol{h}\in U_0$ we have
\begin{gather}\label{roba inutile}
g_{x_0}(0)=0,\quad\quad G_{x_0}(\ol{h},g_{x_0}(\ol{h}))=0.
\end{gather}
Moreover, the function $g_{x_0}:U_0\ra\R$ satisfying \eqref{roba inutile} is uniquely determined. Without loss of generality we may assume that $U_0$ is an open ball centered at the origin of radius $R$. We remark that \eqref{roba inutile} implies that for every $\ol{h}\in h_1^\perp$
\begin{gather}\label{roba inutile 2}
G(x_0+\ol{h}+g_{x_0}(\ol{h})h_1)=G_{x_0}(\ol{h},g_{x_0}(\ol{h}))=0.
\end{gather}

\item[\textbf{Step 2:}] We denote by $D_{h_1^\perp}g_{x_0}(0)$ the Fr\'echet derivative of $g_{x_0}$ at the origin. For $t>0$ sufficiently small and by \eqref{roba inutile 2}, for any $i\geq 2$ we get
\begin{gather*}
0=G_{x_0}(th_i,g_{x_0}(th_i))-G_{x_0}(0,g_{x_0}(0))=\\
=G(x_0+th_i+g_{x_0}(th_i)h_1)-G(x_0+g_{x_0}(0)h_1)=\\
=G\pa{x_0+th_i+g_{x_0}(0)h_1+t\gen{D_{h_1^\perp}g_{x_0}(0),h_i}_{h_1^\perp}h_1+o(t)h_1}-G(x_0+g_{x_0}(0)h_1)=\\
=\gen{\nabla_HG(x_0),th_i+\gen{D_{h_1^\perp}g_{x_0}(0),th_i}_{h_1^\perp}+o(t)h_1}_H.
\end{gather*}
Letting $t$ go to zero, for any $i\geq 2$ we get
\begin{gather}\label{roba inutile 3}
\gen{D_{h_1^\perp}g_{x_0}(0),h_i}_{h_1^\perp}=-\frac{\partial_i G(x_0)}{\partial_1 G(x_0)}.
\end{gather}

\item[\textbf{Step 3:}] The vector field $\Phi_{x_0}(\ol{h})=\sum_{i=2}^{+\infty}\varphi_i(x_0+\ol{h})h_i$ is defined from $h_1^\perp$ to itself. Let $\delta$ be a positive real number
which satisfies
\[\delta\leq \frac{R}{2(K(\Phi)R+\abs{\Phi(x_0)}_H)},\]
where $K(\Phi)$ has been introduced in \eqref{cioccolato}.
We consider the complete metric space $\con_b([-\delta,\delta],\overline{U_0})$, i.e. the set
\[\con_b([-\delta,\delta],\ol{U_0}):=\set{f:\sq{-\delta,\delta}\ra \overline{U_0}\tc f\text{ is continuous}}  
,\]
endowed with the complete metric \(d(f,g):=\sup_{t\in[-\delta,\delta]}\abs{f(t)-g(t)}_{h_1^\perp}\). Let $\Gamma:\con_b([-\delta,\delta],\ol{U_0})\ra \con_b([-\delta,\delta],h_1^\perp)$ be the function defined as follows:
\begin{align}\label{roba strana}
\Gamma(\gamma)(t)=\int_0^t\Phi_{x_0}(\gamma(s))ds,
\end{align}
for any $t\in[-\delta,\delta]$. The integral in \eqref{roba strana} should be understood in the Bochner sense. We look for a fixed point of $\Gamma$ in $\con_b([-\delta,\delta],\ol{U_0})$.  We want to use Banach fixed-point theorem, so
\begin{gather*}
d(\Gamma(\gamma_1),\Gamma(\gamma_2))=\sup_{t\in[-\delta,\delta]}\abs{\Gamma(\gamma_1)(t)-\Gamma(\gamma_2)(t)}_{h_1^\perp}=\displaybreak[0]\\
=\sup_{t\in [-\delta,\delta]}\abs{\int_0^t\Phi_{x_0}(\gamma_1(s))ds-\int_0^t\Phi_{x_0}(\gamma_2(s))ds}_{h_1^\perp}
\leq \sup_{t\in [-\delta,\delta]}\int_0^t\abs{\Phi_{x_0}(\gamma_1(s))-\Phi_{x_0}(\gamma_2(s))}_{h_1^\perp}ds\leq \displaybreak[0]\\
\leq \sup_{t\in [-\delta,\delta]}\int_0^t\abs{\Phi(x_0+\gamma_1(s))-\Phi(x_0+\gamma_2(s))}_{H}ds\leq K(\Phi)\sup_{t\in[-\delta,\delta]}\int_0^t\abs{\gamma_1(s)-\gamma_2(s)}_{h_1^\perp}ds\leq\displaybreak[0]\\
\leq \delta K(\Phi)\sup_{t\in[-\delta,\delta]}\abs{\gamma_1(t)-\gamma_2(t)}_{h_1^\perp}\leq\frac12\sup_{t\in[-\delta,\delta]}\abs{\gamma_1(t)-\gamma_2(t)}_{h_1^\perp}
=\frac12d(\gamma_1,\gamma_2).
\end{gather*}
Therefore $\Gamma$ is a contraction in $\con_b([-\delta,\delta],\ol{U_0})$. We claim that $\Gamma$ maps $\con_b([-\delta,\delta],\ol{U_0})$ into itself. The continuity of $\Gamma(\gamma)(t)$ is clear, and
\begin{gather*}
\sup_{t\in[-\delta,\delta]}\abs{\Gamma(\gamma)(t)}_{h_1^\perp}\leq \sup_{t\in[-\delta,\delta]}\abs{\Gamma(\gamma)(t)-\Gamma(0)(t)}_{h_1^\perp}+\sup_{t\in[-\delta,\delta]}\abs{\Gamma(0)(t)}_{h_1^\perp}\leq\displaybreak[0]\\
\leq \frac R2+\sup_{t\in[-\delta,\delta]}\abs{\int_0^t\Phi_{x_0}(0)ds}_{h_1^\perp}\leq \frac R2+\frac R2=R.
\end{gather*}
By the Banach fixed-point theorem there exists a unique fixed point $\gamma_{x_0}\in\con_b([-\delta,\delta],\ol{U_0})$ of $\Gamma$. We remark that $\gamma_{x_0}(0)=0$ and that, up to replace $\delta>0$ with a smaller one, we can assume that $\gamma_{x_o}([-\delta,\delta])\subseteq U_0$.

\item[\textbf{Step 4:}] We consider the function $\psi_{x_0}:U_0\ra H$, defined as $\psi_{x_0}(\ol{h})=\ol{h}+g_{x_0}(\ol{h})h_1$. We now want to evaluate the function $\sigma_{x_0}:(-\delta,\delta)\ra H$ defined as
\[\sigma_{x_0}(t)=\psi_{x_0}(\gamma_{x_0}(t)),\]
and its derivative at the origin. Observe that
\begin{gather*}
\sigma_{x_0}(t)=\psi_{x_0}(\gamma_{x_0}(t))=\psi_{x_0}\pa{\int_0^t\Phi_{x_0}(\gamma_{x_0}(s))ds}=\\
=\int_0^t\Phi_{x_0}(\gamma_{x_0}(s))ds+g_{x_0}\pa{\int_0^t\Phi_{x_0}(\gamma_{x_0}(s))ds}h_1,
\end{gather*}
so $\sigma_{x_0}(0)=0$. Furthermore
\begin{gather*}
\sigma_{x_0}'(0)=\lim_{t\ra 0}\frac{\sigma(t)-\sigma(0)}{t}=\lim_{t\ra 0}\frac{1}{t}\pa{\int_0^t\Phi_{x_0}(\gamma_{x_0}(s))ds+g_{x_0}\pa{\int_0^t\Phi_{x_0}(\gamma_{x_0}(s))ds}h_1}=\displaybreak[0]\\
=\Phi_{x_0}(\gamma_{x_0}(0))+\lim_{t\ra 0}\frac{1}{t}\pa{g_{x_0}\pa{\int_0^t\Phi_{x_0}(\gamma_{x_0}(s))ds}h_1}=\displaybreak[0]\\
=\Phi_{x_0}(0)+\gen{D_{h_1^\perp}g_{x_0}\pa{0},\Phi_{x_0}(0)}_{h_1^\perp}h_1
=\sum_{i=2}^{+\infty}\varphi_i(x_0)h_i+\pa{\sum_{i=2}^{+\infty}\varphi_i(x_0)\gen{D_{h_1^\perp}g_{x_0}\pa{0},h_i}_{h_1^\perp}}h_1,
\end{gather*}
by \eqref{boh} and \eqref{roba inutile 3} we get
\begin{gather*}
\sigma_{x_0}'(0)=\sum_{i=2}^{+\infty}\varphi_i(x_0)h_i-\pa{\sum_{i=2}^{+\infty}\varphi_i(x_0)\frac{\partial_i G(x_0)}{\partial_1 G(x_0)}}h_1=\sum_{i=1}^{+\infty}\varphi_i(x_0)h_i=\Phi(x_0).
\end{gather*}
We finally claim that for every $t\in(-\delta,\delta)$ we have $G(x_0+\sigma_{x_0}(t))=0$.
Indeed, recalling that $\Gamma(\gamma_{x_0})(t)\in U_0$ and \eqref{roba inutile 2}, we get
\begin{gather*}
G(x_0+\sigma_{x_0}(t))=G(x_0+\psi_{x_0}(\gamma_{x_0}(t)))=G(x_0+\Gamma(\gamma_{x_0})(t)+g_{x_0}(\Gamma(\gamma_{x_0})(t))h_1)=0.
\end{gather*}

\item[\textbf{Step 5:}] Now We are able to prove \eqref{roba inutilissima}. Indeed, from \eqref{bombolone}, \eqref{divergenza su dominio}, $\sigma_{x_0}(0)=0$ and $\sigma_{x_0}'(0)=\Phi(x_0)$ we deduce that
\begin{align*}
\frac d{dt}\nabla_HG(x_0+\sigma_{x_0}(t))_{|t=0}
= & \lim_{t\rightarrow0}\frac{\nabla_HG(x_0+\sigma_{x_0}(t))-\nabla_HG(x_0)}{t} \\
= & \lim_{t\rightarrow0}\frac{\nabla_HG(x_0+\Phi(x_0)t+\Phi(x_0)o(t))-\nabla_HG(x_0)}{t} \\
= & \partial_{\Phi(x_0)}\nabla_H^2G(x_0)=\nabla_H^2G(x_0)(\Phi(x_0)).
\end{align*}
Then, we have
\begin{gather*}
0=\frac{d}{dt}\left(\gen{\Phi(x_0+\sigma_{x_0}(t)),\nabla_H G(x_0+\sigma_{x_0}(t))}_H\right)_{|t=0}=\\
=\gen{\nabla_H\Phi(x_0)\Phi(x_0),\nabla_H G(x_0)}_H
+\gen{\Phi(x_0),\nabla_H^2 G(x_0)\Phi(x_0)}_H.
\end{gather*}
\end{enumerate}
\end{proof}

In the next theorem we prove that the space $Z^{1,2}_U(\Omega,\nu;H)$ is contained in the domain of the divergence, where $Z^{1,2}_U(\Omega,\nu;H)$ is the completion of the space $\mathcal{Z}(\Omega,H)$ with respect to the norm defined in \eqref{norma divergenza su dominio}.

\begin{thm}\label{divergence for W12}
Assume that either Hypotheses \ref{ipotesi dominio} and \ref{ipotesi peso} hold or Hypothesis \ref{ipotesi peso} holds and $\Omega$ is the whole space. Every vector field $\Phi\in Z^{1,2}_U(\Omega,\nu;H)$ has a divergence $\diver_{\nu,\Omega} \Phi\in L^2(\Omega,\nu)$ and for every $f\in W^{1,2}(\Omega,\nu)$, the following equality holds:
\[\int_\Omega\gen{\nabla_H f(x),\Phi(x)}_Hd\nu(x)=-\int_\Omega f(x)\diver_{\nu,\Omega} \Phi(x)d\nu(x).\]
Furthermore, if $\varphi_n=\gen{\Phi,h_n}_H$ for every $n\in\N$ where $(h_n)_{n\in\N}$ is an orthonormal basis of $H$, then
\begin{gather}\label{Formula divergenza}
\diver_{\nu,\Omega} \Phi=\sum_{n=1}^{+\infty}\pa{\partial_n\varphi_n-\varphi_n\partial_nU-\varphi_n\hat{h}_n},
\end{gather}
where the series converges in $\elle^2(\Omega,\nu)$. In addition $\norm{\diver_\nu \Phi}_{\elle^2(\Omega,\nu)}\leq \norm{\Phi}_{Z_U^{1,2}(\Omega,\nu;H)}$.
\end{thm}

\begin{proof}
We prove the theorem assuming Hypotheses \ref{ipotesi dominio} and \ref{ipotesi peso} hold, since the case when Hypothesis \ref{ipotesi peso} holds and $\Omega$ is the whole space can be proved in a similar way. We start with a preliminary computation. Let $\Phi\in \mathcal{Z}(\Omega,H)$, so there exists an orthonormal basis $\set{h_i}_{i\in\N}$ of $H$ such that  $\Phi=\sum_{i=1}^{n} \varphi_ih_i$ for some $n\in\N$ and $\varphi_i\in\fcon^{2}_b(\Omega)$ for every $i=1,\ldots,n$. In addition $\gen{\Phi(x),\nabla_H G(x)}_H=0$ for $\rho$-a.e $x\in G^{-1}(0)$. By the integration by parts formula if $f\in \fcon^\infty_b(\Omega)$ we have
\begin{gather}\label{conti dentro teo divergenza}
\begin{array}{c}
\displaystyle\int_\Omega\gen{\nabla_H f,\Phi}_Hd\nu=\int_\Omega\sum_{i=1}^{n}\partial_i f\varphi_id\nu=\sum_{i=1}^{n}\int_\Omega\partial_i f\varphi_id\nu=\sum_{i=1}^{n}\pa{\int_\Omega\partial_i (f\varphi_i)d\nu-\int_\Omega f\partial_i\varphi_id\nu}=\displaybreak[0]\\
\displaystyle=\sum_{i=1}^{n}\pa{\int_{G^{-1}(0)}f\varphi_i\frac{\trace(\partial_i G)}{\abs{\trace(\nabla_H G)}_H}e^{-\trace(U)}d\rho-\int_\Omega f\pa{\partial_i\varphi_i-\varphi_i\partial_iU-\varphi_i\hat{h}_i}d\nu}=\displaybreak[0]\\
\displaystyle=\int_{G^{-1}(0)}\gen{\Phi,\trace(\nabla_H G)}_H\frac{fe^{-\trace(U)}}{\abs{\trace(\nabla_H G)}_H}d\rho-\sum_{i=1}^{n}\pa{\int_\Omega f\pa{\partial_i\varphi_i-\varphi_i\partial_iU-\varphi_i\hat{h}_i}d\nu}=\displaybreak[0]\\
\displaystyle=-\sum_{i=1}^{n}\pa{\int_\Omega f\pa{\partial_i\varphi_i-\varphi_i\partial_iU-\varphi_i\hat{h}_i}d\nu}.
\end{array}
\end{gather}
So we have
\[\diver_{\nu,\Omega} \Phi=\sum_{i=1}^{n}\pa{\partial_i\varphi_i-\varphi_i\partial_iU-\varphi_i\hat{h}_i}.\]
We recall the definition of the trace operator for nuclear operators $A$. Let $x\in \Omega$ and let $\{h_n\}_{n\in\N}$ be an orthonormal basis of $H$; we say that $A$ is a trace class operator if $\sum_{n=1}^\infty\langle (A^*A)^{1/2}h_n,h_n\rangle_H$ is finite, and we set $\text{trace}_H(A):=\sum_{n=1}^\infty\langle (A^*A)^{1/2}h_n,h_n\rangle_H$. In particular, $(\nabla_H\Phi)^2$ is a trace class operator and \(\text{trace}_H(\nabla_H\Phi(x)^2)\leq \norm{\nabla_H\Phi(x)}^2_{\mathcal{H}_2}\) (see \cite[Appendix A.2]{Bog98}).
By Lemmata \ref{Conti per divergenza con Neumann} and \ref{Lemma derivata seconda G}
\begin{gather}
\displaystyle\int_\Omega\pa{\diver_{\nu,\Omega} \Phi}^2d\nu=\sum_{i=1}^n\int_\Omega\abs{\varphi_i}^2d\nu+\sum_{i=1}^n\sum_{j=1}^n\int_\Omega\partial_j
\varphi_i\partial_i \varphi_jd\nu+\sum_{i=1}^n\sum_{j=1}^n\int_\Omega \varphi_i\varphi_j\partial_i\partial_jUd\nu+\displaybreak[0]\notag\\
\displaystyle+\sum_{i=1}^n\sum_{j=1}^n\int_{G^{-1}(0)}\varphi_i(\partial_j\varphi_j-\varphi_j\trace(\partial_j U)-\varphi_j\hat{h}_j)\frac{\trace(\partial_i G)}{\abs{\trace(\nabla_H G)}}e^{-\trace(U)}d\rho+\displaybreak[0]\notag\\
\displaystyle-\sum_{i=1}^n\sum_{j=1}^n\int_{G^{-1}(0)}\varphi_i\partial_i\varphi_j\frac{\trace(\partial_j G)}{\abs{\trace(\nabla_H G)}_H}e^{-\trace(U)}d\rho=\displaybreak[0]\notag\\
\displaystyle=\norm{\Phi}^2_{\elle^2(\Omega,\nu)}+\int_\Omega\gen{\nabla_H^2U\Phi,\Phi}_Hd\nu+\int_\Omega\text{trace}_H((\nabla_H \Phi)^2)d\nu+\displaybreak[0]\notag\\
\displaystyle+\sum_{j=1}^n\int_{G^{-1}(0)}(\partial_j\varphi_j-\varphi_j\trace(\partial_j U)-\varphi_j\hat{h}_j)\frac{\gen{\Phi,\trace(\nabla_H G)}_H}{\abs{\trace(\nabla_H G)}}e^{-\trace(U)}d\rho+\displaybreak[0]\notag\\
\displaystyle+\int_{G^{-1}(0)}\frac{\gen{\trace(\nabla_H^2 G)\Phi,\Phi}_H}{\abs{\trace(\nabla_H G)}_H}e^{-\trace(U)}d\rho\leq\displaybreak[0]\notag\\
\displaystyle\leq \norm{\Phi}^2_{L^2(\Omega,\nu;H)}+\int_\Omega\norm{\nabla_H \Phi}_{\mathcal{H}_2}^2d\nu+\int_\Omega\gen{\nabla_H^2U\Phi,\Phi}_Hd\nu+\displaybreak[0]\notag\\
\displaystyle+\int_{G^{-1}(0)}\frac{\gen{\trace(\nabla_H^2 G)\Phi,\Phi}_H}{\abs{\trace(\nabla_H G)}_H}e^{-\trace(U)}d\rho=\norm{\Phi}^2_{Z^{1,2}(\Omega,\nu;H)}.\label{ciao}
\end{gather}
Let $(\Phi^n)_{n\in\N}\subseteq\mathcal{Z}(\Omega,H)$ be a sequence of vector fields which converges to $\Phi$ in $Z_U^{1,2}(\Omega,\nu;H)$. By \eqref{ciao}, $(\diver_{\nu,\Omega} \Phi^n)$ is a Cauchy sequence in $\elle^2(\Omega,\nu)$ and therefore it converges to an element of $\elle^2(\Omega,\nu)$ which we denote by $\diver_{\nu,\Omega} \Phi$. By formula \eqref{conti dentro teo divergenza}, it is easily seen that $\diver_{\nu,\Omega}\Phi$ satisfies (\ref{condizione di divergenza in H}). Finally, by a standard approximation argument
we can conclude that $\diver_{\nu,\Omega}\Phi$ fulfills (\ref{condizione di divergenza in H}) also for every $f\in W^{1,2}(\Omega,\nu)$.
\end{proof}

We say that a subspace $S$ of $W^{1,2}(\Omega,\nu;H)$, endowed with a Banach norm $\norm{\cdot}_S$, is a \emph{Neumann extension subspace}
if any $\Phi\in S$ satisfies $\langle \Phi,\nabla_H G\rangle_H=0$ $\rho$-a.e. on $G^{-1}(0)$ and it admits a continuous linear extension operator, i,e., 
if there exists a linear operator $E_S: S\ra Z_U^{1,2}(X,\nu,H)$ such that for every $\Phi\in S$
\begin{enumerate}
\item $E_S \Phi(x)=\Phi(x)$ and $\nabla_H E_S\Phi(x)=\nabla_H\Phi(x)$ for $\mu$-a.e $x\in \Omega$;

\item there is $K_S>0$, independent of $\Phi$, such that $\norm{E_S\Phi}_{Z_U^{1,2}(X,\mu;H)}\leq K_S\norm{\Phi}_{S}$.
\end{enumerate}
As a corollary of Theorem \ref{divergence for W12} we get the following.

\begin{cor}\label{divext}
Assume that Hypotheses \ref{ipotesi dominio} and \ref{ipotesi peso} hold and let $S$ be a Neumann extension subspace with norm $\norm{\cdot}_S$. Every field $\Phi\in S$ has a divergence $\diver_{\nu,\Omega} \Phi\in \elle^2(\Omega,\nu)$ and for every $f\in W^{1,2}(\Omega,\nu)$, the following equality holds:
\[\int_\Omega\gen{\nabla_H f(x),\Phi(x)}_Hd\nu(x)=-\int_\Omega f(x)\diver_{\nu,\Omega} \Phi(x)d\nu(x).\]
Furthermore, if $\varphi_n=\gen{\Phi,h_n}_H$ for every $n\in\N$, where $(h_n)_{n\in\N}$ is an orthonormal basis of $H$, then
\begin{gather*}
\diver_{\nu,\Omega} \Phi=\sum_{n=1}^{+\infty}\pa{\partial_n\varphi_n-\varphi_n\partial_nU-\varphi_n\hat{h}_n},
\end{gather*}
where the series converges in $\elle^2(\Omega,\nu)$. In addition, $\norm{\diver_\nu \Phi}_{\elle^2(\Omega,\nu)}\leq K_S\norm{\Phi}_{S}$.
\end{cor}

\begin{proof}
Let us consider the divergence $\diver_{\nu,X}E_S\Phi$ (Theorem \ref{divergence for W12}). For $\nu$-a.e. every $x\in\Omega$ let
\[D_k(x):=\sum_{n=1}^{k}\pa{\partial_n \varphi_n(x)-\varphi_n(x)\partial_nU(x)-\varphi_n(x)\hat{h}_n(x)}.\]
We have that
\begin{gather}\label{stanco}
\int_\Omega\abs{D_k-D_m}^2d\nu\leq \int_X\abs{\sum_{n=k+1}^m\partial_nE_S\varphi_n-E_S\varphi_n\partial_nU-E_S\varphi_n\hat{h}_n}^2d\nu,
\end{gather}
where $E_S\varphi_n:=\langle E_S\Phi,h_n\rangle_H$.
Since the right hand side of \eqref{stanco} converges to zero (the series converges to $\diver_{\nu,X}E_S\Phi$) we get that $(D_k)_{k\in\N}$ is a Cauchy sequence in $\elle^2(\Omega,\nu)$. We denote by $D_\infty\Phi$ the limit of $D_n$ in $L^2(\Omega,\nu)$ and we observe that for every $f\in W^{1,2}(\Omega,\nu)$
\begin{gather*}
\int_\Omega\gen{\nabla_Hf,\Phi}_Hd\nu=\lim_{n\ra+\infty}\sum_{i=1}^n\int_\Omega\partial_if\varphi_id\nu=\\
=\lim_{n\ra+\infty}\sum_{i=1}^n\pa{\int_\Omega f\pa{\varphi_i(\partial_iU+\hat{h}_i)-\partial_i\varphi_i}d\nu+\int_{G^{-1}(0)}f\varphi_i\frac{\partial_i G}{\abs{\nabla_H G}_H}e^{-U}d\rho}.
\end{gather*}
We remark that $\rho$-a.e we have
\[\sum_{i=1}^nf\varphi_i\frac{\partial_i G}{\abs{\nabla_H G}_H}e^{-U}\longra f\gen{\Phi,\nabla_H G}_H\frac{e^{-U}}{\abs{\nabla_H G}_H}=0,\]
and
\[\abs{\sum_{i=1}^nf\varphi_i\frac{\partial_i G}{\abs{\nabla_H G}_H}e^{-U}}\leq \abs{f}e^{-U}\abs{\Phi}_H\in\elle^1(G^{-1}(0),\rho).\]
Therefore, by the Lebesgue's dominated convergence theorem and the continuity of the trace operator (Proposition \ref{trace continuity}) we get $\int_\Omega\gen{\nabla_Hf,\Phi}_Hd\nu=-\int_\Omega fD_\infty\Phi d\nu$ for any $f\in W^{1,2}(\Omega,\nu)$.
This means that $\diver_{\nu,\Omega}\Phi$ exists and $\diver_{\nu,\Omega}\Phi=D_\infty\Phi$. Moreover
\begin{gather*}
\norm{\diver_{\nu,\Omega}\Phi}_{\elle^2(\Omega,\nu)}\leq\liminf_{k\ra+\infty}\norm{D_k}_{\elle^2(\Omega,\nu)}\leq \liminf_{k\ra+\infty}\norm{D_k}_{\elle^2(X,\nu)}=\\ 
=\norm{\diver_{\nu,X}E_S\Phi}_{\elle^2(X,\nu)}\leq \norm{E_S\Phi}_{Z_U^{1,2}(X,\nu;H)}\leq K_S\norm{\Phi}_{S}.
\end{gather*}
\end{proof}

\begin{remark}
The subspace of the vector fields $\Phi\in Z_U^{1,2}(\Omega,\nu;H)$ such that the extension
\[\tilde{\Phi}(x):=\eqsys{\Phi(x) & x\in\Omega;\\ 0 & x\notin\Omega,}\]
belongs to $Z_U^{1,2}(X,\nu;H)$ satisfies the hypotheses of Corollary \ref{divext}.
\end{remark}

\section{Maximal Sobolev regularity}\label{Maximal Sobolev regularity}

This Section is devoted to the the study of maximal Sobolev regularity for the equation
\begin{gather}\label{Problema}
\lambda u(x)-L_{\nu,\Omega}u(x)=f(x)\qquad \mu\text{-a.e. }x\in \Omega,
\end{gather}
where $\lambda >0$, and $f\in\elle^2(\Omega,\nu)$, since a part of the proofs of Theorems \ref{cor whole space}, \ref{cor halfspaces} and \ref{thm extension}
relies on them. The results of this section are sharper than the results contained in \cite{CF16} and \cite{CF16convex}.

Our main result is the following theorem.

\begin{thm}\label{Main theorem 1}
Assume that Hypotheses \ref{ipotesi dominio} and \ref{ipotesi peso} hold. For every $\lambda>0$ and $f\in\elle^2(\Omega,\nu)$ problem (\ref{Problema}) has a unique weak solution $u\in W^{2,2}_U(\Omega,\nu)$. In addition the following hold
\begin{gather}
\label{Neumann condition}\gen{\nabla_H u(x),\nabla_H G(x)}_H=0\text{ for }\rho\text{-a.e. }x\in G^{-1}(0);\\
\label{1 stime max}
\norm{u}_{\elle^2(\Omega,\nu)}\leq\frac{1}{\lambda}\norm{f}_{\elle^2(\Omega,\nu)};\qquad \norm{\nabla_H u}_{\elle^2(\Omega,\nu;H)}\leq\frac{1}{\sqrt{\lambda}}\norm{f}_{\elle^2(\Omega,\nu)};\\
\label{2 stime max}
\norm{\nabla_H^2 u}_{\elle^2(\Omega,\nu;\mathcal{H}_2)}^2+\int_\Omega\gen{\nabla_H^2 U\nabla_H u,\nabla_H u}d\nu\leq 2\norm{f}^2_{\elle^2(\Omega,\nu)}.
\end{gather}
In particular $u\in W^{2,2}_{U,N}(\Omega,\nu)$.
\end{thm}

We split the proof of Theorem \ref{Main theorem 1} into two parts: in the Section \ref{cacca1} we study the case of $\Omega=X$ and $U$ with $H$-Lipschitz gradient, in Section \ref{cacca2} we use the results of Section \ref{cacca1} to prove Theorem \ref{Main theorem 1}.

\subsection{$\Omega$ is the whole space}\label{cacca1}

We start this subsection assuming the following hypothesis on the weight:

\begin{hyp}\label{ipotesi peso 2}
Let $U:X\ra \R$ be a function satisfying Hypothesis \ref{ipotesi peso}. Assume that $U$ is differentiable along $H$ at every point $x\in X$, and $\nabla_H U$ is $H$-Lipschitz.
\end{hyp}
\noindent We remark that every convex function in $\fcon^2_b(X)$ and every continuous linear functional $x^*\in X^*$ satisfy Hypothesis \ref{ipotesi peso 2}.

We will recall some results about maximal Sobolev regularity contained in \cite{CF16}. Let us consider the problem
\begin{gather}\label{Problema su spazio}
\lambda u(x)-L_{\nu}u(x)=f(x)\qquad \mu\text{-a.e. }x\in X,
\end{gather}
where $\lambda >0$, $f\in\elle^2(X,\nu)$, and $L_\nu:=L_{\nu,X}$. A function $u\in D(L_\nu)$ of problem \eqref{Problema su spazio} is said to be a \emph{strong solution} if there exists a sequence $\{u_n\}_{n\in\N}\subseteq \fcon^3_b(X)$ such that $u_n$ converges to $u$ in $\elle^2(X,\nu)$ and
\[\elle^2(X,\nu)\text{-}\lim_{n\ra+\infty}\lambda u_n-L_\nu u_n=f.\]
Moreover a sequence $\set{u_n}_{n\in\N}\subseteq\fcon^3_b(X)$ satisfying the above conditions is called a \emph{strong solution sequence for $u$}. The following proposition is borrowed from \cite[Proposition 5.8]{CF16}.

\begin{thm}\label{existence of strong solution}
Assume that Hypothesis \ref{ipotesi peso 2} holds. For every $\lambda >0$ and $f\in\elle^2(X,\nu)$, there exists a unique strong solution of equation (\ref{Problema su spazio}). Such strong solution is also a weak solution of problem \eqref{Problema su spazio}. In addition, if $\set{u_n}_{n\in\N}\subseteq\fcon^3_b(X)$ is a strong solution sequence for $u$, then $(u_n)$ converges to $u$ in $W^{2,2}(X,\nu)$.
\end{thm}
When $U$ satisfies Hypothesis \ref{ipotesi peso 2} we have the following regularity result.

\begin{thm}\label{Stime per lip}
Let $U$ be a function satisfying Hypothesis \ref{ipotesi peso 2}, let $\lambda>0$, $f\in\elle^2(X,\nu)$, and let $u$ be the strong solution of equation (\ref{Problema su spazio}).
Then $u\in W^{2,2}_U(X,\nu)$ and
\begin{gather}
\label{1 stime max per lip}\norm{u}_{\elle^2(X,\nu)}\leq\frac{1}{\lambda}\norm{f}_{\elle^2(X,\nu)};\qquad \norm{\nabla_H u}_{\elle^2(X,\nu;H)}\leq\frac{1}{\sqrt{\lambda}}\norm{f}_{\elle^2(X,\nu)};\\
\label{2 stime max per lip}\norm{\nabla_H^2 u}_{\elle^2(X,\nu;\mathcal{H}_2)}^2+\int_X\gen{\nabla_H^2 U\nabla_H u,\nabla_H u}_Hd\nu\leq 2\norm{f}^2_{\elle^2(X,\nu)}.
\end{gather}
\end{thm}
\noindent
The difference between Theorem \ref{Stime per lip} and the results of \cite{CF16} is that estimate \eqref{2 stime max per lip} is sharper, since it contains the integral $\int_X\gen{\nabla_H^2 U\nabla_H u,\nabla_H u}_Hd\nu$. We stress that, even if $\nabla_HU$ is $H$-Lipschitz, which means that $\nabla^2_HU$ is essentially bounded, we can not use the second inequality in \eqref{1 stime max per lip} to estimate \eqref{2 stime max per lip}. Indeed, \eqref{2 stime max per lip} is independent of $\lambda$, while \eqref{1 stime max per lip} does not.

\begin{proof}
The proof of \eqref{1 stime max per lip} can be found in \cite[Theorem 5.10]{CF16}.
By Proposition \ref{existence of strong solution} there exists a sequence $\set{u_n}_{n\in\N}\subseteq \fcon^3_b(X)$ and a function $u\in W^{1,2}(X,\nu)$ such that $u_n$ converges to $u$ in $\elle^2(X,\nu)$ and
\[\elle^2(X,\nu)\text{-}\lim_{n\ra+\infty}\lambda u_n-L_\nu u_n=f.\]
Let $f_n:=\lambda u_n-L_\nu u_n$. Using formula \eqref{formula Lnu}, we differentiate the equality $\lambda u_n-L_\nu u_n=f_n$ with respect to the $e_j$ direction, multiply the result by $\partial_j u$, sum over $j$ and finally integrate over $X$ with respect to $\nu$. Then we obtain
\begin{gather*}
(1+\lambda)\int_X\abs{\nabla_H u_n}_H^2d\nu+\int_X\norm{\nabla_H^2 u_n}_{\mathcal{H}_2}^2d\nu+\int_X\gen{\nabla_H^2 U\nabla_H u_n,\nabla_H u_n}_Hd\nu=\\
=\int_Xf_n^2d\nu-\lambda\int_Xf_nu_nd\nu.
\end{gather*}
By Fatou's Lemma and recalling that $u_n$ and $f_n$ converge to $u$ and $f$ in $\elle^2(X,\nu)$, respectively, we get
\begin{gather*}
\norm{\nabla_H^2 u}_{\elle^2(X,\nu;\mathcal{H}_2)}^2+\int_X\gen{\nabla_H^2 U\nabla_H u,\nabla_H u}_Hd\nu\leq\\
\leq\liminf_{n\ra+\infty}\pa{\norm{\nabla_H^2 u}_{\elle^2(X,\nu;\mathcal{H}_2)}^2+\int_X\gen{\nabla_H^2 U\nabla_H u,\nabla_H u}_Hd\nu}\leq \liminf_{n\ra+\infty} \pa{\int_Xf_n^2d\nu-\lambda\int_Xf_nu_nd\nu}=\\
=\int_Xf^2d\nu-\lambda\int_Xfud\nu.
\end{gather*}
Using inequalities \eqref{1 stime max per lip} we get
\[\norm{\nabla_H^2 u}_{\elle^2(X,\nu;\mathcal{H}_2)}^2+\int_X\gen{\nabla_H^2 U\nabla_H u,\nabla_H u}_Hd\nu\leq 2\norm{f}^2_{\elle^2(X,\nu)}.\]
\end{proof}

We will not give the prove of the following theorem, since it can be easily deduced using the results of \cite{CF16} and the arguments in the proof of Theorem \ref{Main theorem 1}.

\begin{thm}\label{Stime tutto spazio}
Assume Hypothesis \ref{ipotesi peso} holds. Let $\lambda>0$, $f\in\elle^2(X,\nu)$, and let $u$ be the strong solution of equation (\ref{Problema su spazio}). Then $u\in W^{2,2}_U(X,\nu)$ and
\begin{gather*}
\norm{u}_{\elle^2(X,\nu)}\leq\frac{1}{\lambda}\norm{f}_{\elle^2(X,\nu)};\qquad \norm{\nabla_H u}_{\elle^2(X,\nu;H)}\leq\frac{1}{\sqrt{\lambda}}\norm{f}_{\elle^2(X,\nu)};\\
\notag\norm{\nabla_H^2 u}_{\elle^2(X,\nu;\mathcal{H}_2)}^2+\int_X\gen{\nabla_H^2 U\nabla_H u,\nabla_H u}d\nu\leq 2\norm{f}^2_{\elle^2(X,\nu)}.
\end{gather*}
\end{thm}

\subsection{The general case}\label{cacca2}

Assume that Hypotheses \ref{ipotesi dominio} and \ref{ipotesi peso} hold. Let $x\in X$ and let $\mathcal{C}\subseteq X$ be a Borel set. We define
\begin{gather*}
d_H(x,\mathcal{C}):=\eqsys{\inf\set{\abs{h}_H\tc h\in H\cap (x-\mathcal{C})} & \text{ if }H\cap(x-\mathcal{C})\neq\emptyset;\\
+\infty & \text{ if }H\cap(x-\mathcal{C})=\emptyset.}
\end{gather*}
$d_H$ can be seen as a distance function from $\mathcal{C}$ along $H$. This function has been already considered in \cite{Kus82}, \cite{UZ97}, \cite[Example 5.4.10]{Bog98}, \cite{Hin11}, and \cite{CF16convex}.
For $\alpha\in(0,1]$ let $U_\alpha$ be the Moreau--Yosida approximation along $H$ of the weight $U$ defined in Section \ref{Second-order analysis of the Moreau--Yosida approximations along $H$}.

We approach the problem in $\Omega$ by penalized problems in the whole space $X$, replacing $U$ by
\begin{gather*}
V_\alpha(x):=U_\alpha(x)+\frac{1}{2\alpha}d_H^2(x,\Omega).
\end{gather*}
for $\alpha\in(0,1]$. Namely for $\alpha\in(0,1]$, we consider the problem
\begin{gather}\label{problema con alpha}
\lambda u_\alpha-L_{\nu_\alpha}u_\alpha=f
\end{gather}
where $\lambda>0$, $f\in\elle^2(X,\nu_\alpha)$, $\nu_\alpha=e^{-V_\alpha}\mu$ and $L_{\nu_\alpha}:=L_{\nu_\alpha,X}$. The first result we need to recall is \cite[Proposition 5.2]{CF16convex}.

\begin{pro}\label{check proprieta V alpha}
Assume that Hypotheses \ref{ipotesi dominio} and \ref{ipotesi peso} hold and let $\alpha\in(0,1]$. Then the following properties hold:
\begin{enumerate}
\item $V_\alpha$ is a convex and $H$-continuous function;

\item $V_\alpha$ is differentiable along $H$ for $\mu$-a.e. $x\in X$, and $\nabla_H V_\alpha$ $H$-Lipschitz;\label{Valpha gradiente Hlip}

\item $e^{-V_\alpha}\in W^{1,p}(X,\mu)$, for every $p\geq 1$;

\item $V_\alpha\in W^{2,t}(X,\mu)$, where $t$ is given by Hypothesis \ref{ipotesi peso};

\item $\lim_{\alpha\ra 0^+}V_\alpha(x)=\eqsys{U(x) & x\in\Omega;\\ +\infty & x\notin \Omega.}$\label{convergenza V alpha}
\end{enumerate}
\end{pro}

By Proposition \ref{check proprieta V alpha} we can apply Theorem \ref{Stime per lip} to problem \eqref{problema con alpha} and get the following maximal Sobolev regularity result (see also \cite[Theorem 5.3]{CF16convex}).

\begin{thm}\label{Stime per lip con alpha}
Assume Hypotheses \ref{ipotesi dominio} and \ref{ipotesi peso} hold and let $\alpha\in(0,1]$, $\lambda>0$ and $f\in\elle^2(X,\nu_\alpha)$. Equation \eqref{problema con alpha} has a unique weak solution $u_\alpha$. Moreover $u_\alpha\in W^{2,2}_{V_\alpha}(X,\nu_\alpha)$ and
\begin{gather}
\label{1 stime max per lip con alpha}\norm{u_\alpha}_{\elle^2(X,\nu_\alpha)}\leq\frac{1}{\lambda}\norm{f}_{\elle^2(X,\nu_\alpha)};\qquad \norm{\nabla_H u_\alpha}_{\elle^2(X,\nu_\alpha;H)}\leq\frac{1}{\sqrt{\lambda}}\norm{f}_{\elle^2(X,\nu_\alpha)};\\
\label{2 stime max per lip con alpha}\norm{\nabla_H^2 u_\alpha}_{\elle^2(X,\nu_\alpha;\mathcal{H}_2)}^2+\int_X\gen{\nabla_H^2 V_\alpha\nabla_H u_\alpha,\nabla_H u_\alpha}d\nu_\alpha\leq 2\norm{f}^2_{\elle^2(X,\nu_\alpha)}.
\end{gather}
In addition, for every $\alpha\in(0,1]$, there exists a sequence $\{u_{\alpha}^{(n)}\}_{n\in\N}\subseteq \fcon^3_b(X)$ such that $u_{\alpha}^{(n)}$ converges to $u_\alpha$ in $W^{2,2}(X,\nu_\alpha)$ and $\lambda u_{\alpha}^{(n)}-L_{\nu_\alpha} u_{\alpha}^{(n)}$ converges to $f$ in $\elle^2(X,\nu_\alpha)$.
\end{thm}

We are now ready to prove Theorem \ref{Main theorem 1}.

\begin{proof}[Proof of Theorem \ref{Main theorem 1}]
The Neumann condition \eqref{Neumann condition} and estimates \eqref{1 stime max} have been proved in \cite[Theorems 1.3 and 1.4]{CF16convex}.
Hence, it remains to prove \eqref{2 stime max}. Let $f\in\fcon^\infty_b(X)$. By Theorem \ref{Stime per lip con alpha}, for every $\alpha\in(0,1]$ the equation \eqref{problema con alpha} has a unique weak solution $u_\alpha\in W^{2,2}(X,\nu_\alpha)$ such that inequalities \eqref{1 stime max per lip con alpha} and \eqref{2 stime max per lip con alpha} hold. Moreover, for every $\varphi\in\fcon^\infty_b(X)$ we have
\begin{gather*}
\lambda \int_X u_\alpha\varphi d\nu_\alpha+\int_X\gen{\nabla_H u_\alpha,\nabla_H \varphi}_Hd\nu_\alpha=\int_Xf\varphi d\nu_\alpha.
\end{gather*}
By Proposition \ref{check proprieta V alpha} and Proposition \ref{proprieta MY}\eqref{convergenza MY} we have
\begin{gather*}
e^{-U(x)}\leq e^{-U_\alpha(x)}= e^{-V_\alpha(x)},\qquad x\in \Omega,
\end{gather*}
and so the inclusion $W^{2,2}(\Omega,\nu_\alpha)\subseteq W^{2,2}(\Omega,\nu)$ follows, for every $\alpha\in(0,1]$.

Let $\set{\alpha_n}_{n\in\N}$ be a sequence converging to zero such that $0< \alpha_n\leq 1$ for every $n\in\N$. By inequalities \eqref{1 stime max per lip con alpha} and \eqref{2 stime max per lip con alpha} the sequence $\set{u_{\alpha_n}\tc n\in\N}$ is bounded in $W^{2,2}(\Omega,\nu)$. By weak compactness there exists a subsequence, that we will still denote by $\set{\alpha_n}_{n\in\N}$, such that $u_{\alpha_n}$ weakly converges to an element $u\in W^{2,2}(\Omega,\nu)$. Without loss of generality we can assume that $u_{\alpha_n}$, $\nabla_H u_{\alpha_n}$ and $\nabla_H^2 u_{\alpha_n}$ converge pointwise $\mu$-a.e. respectively to $u$, $\nabla_H u$ and $\nabla_H^2 u$. By Fatou's lemma and inequality \eqref{2 stime max per lip con alpha} we get
\begin{gather*}
\norm{\nabla_H^2 u}_{\elle^2(\Omega,\nu;\mathcal{H}_2)}^2+\int_\Omega\gen{\nabla_H^2 U\nabla_H u,\nabla_H u}d\nu\leq\\
\leq\liminf_{n\ra+\infty}\pa{\norm{\nabla_H^2 u_{\alpha_n}}_{\elle^2(\Omega,\nu_{\alpha_n};\mathcal{H}_2)}^2+\int_\Omega\gen{\nabla_H^2 U_{\alpha_n}\nabla_H u_{\alpha_n},\nabla_H u_{\alpha_n}}d\nu_{\alpha_n}}\leq\\
\leq\liminf_{n\ra+\infty}\pa{\norm{\nabla_H^2 u_{\alpha_n}}_{\elle^2(X,\nu_{\alpha_n};\mathcal{H}_2)}^2+\int_X\gen{\nabla_H^2 V_{\alpha_n}\nabla_H u_{\alpha_n},\nabla_H u_{\alpha_n}}d\nu_{\alpha_n}}\leq\\
\leq 2\norm{f}^2_{\elle^2(X,\nu_{\alpha_n})}\leq 2\norm{f}^2_{\elle^2(\Omega,\nu)}.
\end{gather*}
Finally, if $f\in\elle^2(\Omega,\nu)$, a standard density argument gives us the assertions of our theorem.
\end{proof}

\section{Proof of the main results and some corollaries}\label{Proof of the main result and some observations}

Theorems \ref{cor whole space}, \ref{cor halfspaces} and \ref{thm extension} are consequence of the following result.

\begin{thm}\label{Main theorem}
Assume that either Hypotheses \ref{ipotesi dominio} and \ref{ipotesi peso} hold or Hypothesis \ref{ipotesi peso} holds and $\Omega$ is the whole space. Then $Z_{U,N}^{2,2}(\Omega,\nu)\subseteq D(L_{\nu,\Omega})\subseteq W_{U,N}^{2,2}(\Omega,\nu)$. Furthermore if we denote with $\norm{\cdot}_{D(L_{\nu,\Omega})}$ the graph norm in $D(L_{\nu,\Omega})$, i.e. for $u\in D(L_{\nu,\Omega})$
\begin{gather*}
\norm{u}_{D(L_{\nu,\Omega})}^2:=\norm{u}_{\elle^2(\Omega,\nu)}^2+\norm{L_{\nu,\Omega} u}^2_{\elle^2(\Omega,\nu)},
\end{gather*}
then for $u\in Z_U^{2,2}(\Omega,\nu)$ and $v\in D(L_{\nu,\Omega})$ it holds that
\begin{gather*}
\norm{u}_{D(L_{\nu,\Omega})}\leq \norm{u}_{Z^{2,2}_{U,N}(\Omega,\nu)}\qquad\text{and} \qquad\norm{v}_{W^{2,2}_{U,N}(\Omega,\nu)}\leq 2\sqrt{2}\norm{v}_{D(L_{\nu,\Omega})}.
\end{gather*}
\end{thm}

\begin{proof}
We prove the theorem assuming Hypotheses \ref{ipotesi dominio} and \ref{ipotesi peso} hold, since in the case when Hypothesis \ref{ipotesi peso} holds and $\Omega$ is the whole space the proof can be obtained in a similar way using Theorem \ref{Stime tutto spazio}.

Let $u\in D(L_{\nu,\Omega})$. Hence, \(\lambda u-L_\nu u\in\elle^2(\Omega,\nu)\), for every $\lambda\in (0,1)$, and by Theorem \ref{Main theorem 1} we get $u\in W^{2,2}_{U,N}(\Omega,\nu)$. Moreover
\begin{gather}\label{equivalenza norme 2}
\begin{array}{c}
\displaystyle\norm{u}_{W^{2,2}_{U,N}(\Omega,\nu)}^2\leq\pa{\frac{1}{\lambda^2}+\frac{1}{\lambda}+2}\norm{\lambda u- L_{\nu,\Omega} u}_{\elle^2(\Omega,\nu)}^2\leq\\
\displaystyle \leq \pa{\frac{1}{\lambda^2}+\frac{1}{\lambda}+2}\pa{2\lambda^2\norm{u}_{\elle^2(\Omega,\nu)}^2+2\norm{L_{\nu,\Omega} u}_{\elle^2(\Omega,\nu)}^2}\leq \\
\displaystyle \leq 2\pa{\frac{1}{\lambda^2}+\frac{1}{\lambda}+2}\pa{\norm{u}_{\elle^2(\Omega,\nu)}^2+\norm{L_{\nu,\Omega} u}_{\elle^2(\Omega,\nu)}^2}=2\pa{\frac{1}{\lambda^2}+\frac{1}{\lambda}+2}\norm{u}^2_{D(L_{\nu,\Omega})}.
\end{array}
\end{gather}
Letting $\lambda\ra 1^-$ in inequality \eqref{equivalenza norme 2} we get $\norm{u}_{W^{2,2}_U(\Omega,\nu)}\leq 2\sqrt{2}\norm{u}_{D(L_{\nu,\Omega})}$.

Assume that $u\in Z^{2,2}_{U,N}(\Omega,\nu)$. Proposition \ref{divergence for W12} implies that $\diver_{\nu,\Omega}\nabla_H u\in \elle^2(\Omega,\nu)$ and
\[\int_\Omega\gen{\nabla_H f,\nabla_H u}_Hd\nu=-\int_\Omega f\diver_{\nu,\Omega}\nabla_H ud\nu,\]
for every $f\in\fcon^\infty_b(\Omega)$. Then we have $u\in D(L_{\nu,\Omega})$ and $L_{\nu,\Omega} u=\diver_{\nu,\Omega}\nabla_H u$. By Proposition \ref{divergence for W12} we have
\begin{gather*}
\norm{u}_{D(L_{\nu,\Omega})}^2=\norm{u}_{\elle^2(\Omega,\nu)}^2+\norm{L_{\nu,\Omega} u}_{\elle^2(\Omega,\nu)}^2=\norm{u}_{\elle^2(\Omega,\nu)}^2+\norm{\diver_{\nu,\Omega}\nabla_H u}_{\elle^2(\Omega,\nu)}^2\leq\\
\leq \norm{u}^2_{\elle^2(\Omega,\nu)}+\norm{\nabla_H u}^2_{Z^{1,2}_U(\Omega,\nu;H)}=\norm{u}^2_{Z^{2,2}_{U,N}(\Omega,\nu)},
\end{gather*}
for every $u\in Z^{2,2}_{U,N}(\Omega,\nu)$.
\end{proof}

We can actually simplify the statement of Theorem \ref{cor whole space} when $\nabla_H U$ is $H$-Lipschitz and $\Omega=X$. Indeed, let us observe that if $\nabla_H U$ is $H$-Lipschitz then the function $x\mapsto\|\nabla_H^2 U(x)\|_{\mathcal{H}_2}$ is essentially bounded (see \cite[Theorem 5.11.2(ii)]{Bog98}). So $W^{2,2}(X,\nu)$ is isomorphic to $W^{2,2}_U(X,\nu)$, with
\[\norm{u}_{W^{2,2}(X,\nu)}\leq \norm{u}_{W^{2,2}_U(X,\nu)}\leq \max\{1,\esssup_{x\in X}\norm{\nabla_H^2 U(x)}_{\mathcal{H}_2}\}\norm{u}_{W^{2,2}(X,\nu)}.\]
In particular if $\nabla_H U$ is $H$-Lipschitz, then $\fcon^2_b(X)$ is dense in $W^{2,2}_U(X,\nu)$.

\begin{cor}\label{corollario lip}
Assume Hypothesis \ref{ipotesi peso} holds and $\nabla_H U$ is $H$-Lipschitz. Then $D(L_{\nu,X})= W^{2,2}(X,\nu)$. Moreover, for every $u\in D(L_{\nu,X})$, it holds $L_{\nu,X} u=\diver_{\nu,X}\nabla_H u$ and
\begin{gather*}
\frac{1}{\max\{1,\esssup_{x\in X}\norm{\nabla_H^2 U(x)}_{\mathcal{H}_2}\}}\norm{u}_{D(L_{\nu,X})}\leq\norm{u}_{W^{2,2}(X,\nu)}\leq 2\sqrt{2}\norm{u}_{D(L_{\nu,X})}.
\end{gather*}
The same holds true, with obvious modifications, when $\Omega$ is a Neumann extension domain.
\end{cor}
\noindent This result has been already proved in \cite[Theorem 6.2]{CF16}.

\section{Examples}\label{Examples}

We conclude the paper by presenting some examples. In Subsection \ref{The unit sphere of Hilbert spaces} we study in detail the case when $\Omega$ is the ball sphere of a Hilbet space and we show that, in this case, the spaces $\mathcal{Z}(\Omega,H)$ is non-trivial, namely it is infinite dimensional, but the space $\mathcal{Z}(\Omega)$ contains only the constant functions. In Subsection \ref{The Ornstein--Uhlenbeck operator on half-spaces} we prove Theorem \ref{thm halfspaces} giving a characterization of the domain of the Ornstein--Uhlenbeck operator on half-spaces.

\subsection{The unit sphere of a Hilbert space}\label{The unit sphere of Hilbert spaces}

Let $X$ be a separable Hilbert space, with norm $\norm{\cdot}_X$ and inner product $(\cdot,\cdot)_X$, and let $\mu$ be a centered non-degenerate Gaussian measure on $X$. Let $\{h_n\}_{n\in\N}$ be an orthonormal basis of $X$ which consists of eigenvector of the covariance operator $Q$, i.e. $Qh_n=\lambda_ih_n$, it is known that an orthonormal basis of the Cameron--Martin space $H$ is $\{\sqrt{\lambda_n}h_n\}_{n\in\N}$ (see \cite{Bog98}).

Consider $G(x)=(x,x)_X-1$, for any $x\in X$, then
\[\Omega=\set{x\in X\tc \norm{x}_X\leq 1}=:\ball_X.\]
Clearly, $G(x)=0$ if and only if $x\in\sfera_X$ the unit sphere of $X$.
Moreover, easy computations show that $\partial_h G(x)=2( x,h)_X$ for any $x\in X$ and any $h\in H$. Hence, if , we have
\[\nabla_HG(x)=2Q^{1/2}x=2\sum_{n=1}^\infty\sqrt{\lambda_n}( x,h_n)_X(\sqrt{\lambda_n}h_n),\]
for any $x\in X$, and so $|\nabla_HG(x)|^2_H=4\|Q^{1/2}x\|_X^2=4\sum_{n=1}^\infty\lambda_n(x,h_n)_X^2$. So $\nabla_H G(x)=0$ if, and only if, $x=0$. Finally $G$ satisfies Hypothesis \ref{ipotesi dominio}\eqref{ipo dominio non degeneratezza}-\eqref{ipo dominio per dini} (see \cite{CF16convex}) and $\partial_{n,m}G(x)=2\lambda_n\delta_{n,m}$.

As an admissible weight we can take $U(x):=\Phi(\|x\|_X^2)$, where $\Phi:\R\longrightarrow\R$ is a $C^2$ convex function which satisfies
\begin{align*}
|\Phi'(t)|,|\Phi''(t)|\leq t^k,\qquad t\in\R,
\end{align*}
for some positive integer $k$. It is easy to prove that $U$ is convex and satisfies the Hypothesis \ref{ipotesi peso}.

Observe that
\begin{gather*}
\mathcal{Z}(\ball_X,H):=\set{\Phi:\ball_X\ra H \tc\begin{array}{c}
\text{there exists $n\in\N$ and $\set{k_1,\ldots, k_n}\subseteq H$}\\
\text{such that $\Phi=\sum_{i=1}^{n}\varphi_i k_i$ for some $n\in\N$,}\\
\text{and $\varphi_i\in\fcon_b^2(\Omega)$ for $i=1,\ldots,n$.}\\
\text{In addition $\gen{\Phi,\nabla_H G}_H=0$ $\rho$-a.e. in $\sfera_X$.}
\end{array}}.
\end{gather*}
In particular all the vector fields
\[\Phi_{i,j}(x)=-\frac{(x,h_i)_X}{\sqrt{\lambda_i}}h_j+\frac{(x,h_j)_X}{\sqrt{\lambda_j}}h_i\]
belongs to $\mathcal{Z}(\ball_X,H)$, so the space $Z^{1,2}_U(\ball_X,\nu;H)$ is infinite dimensional and contained in the domain of the divergence operator (see Theorem \ref{divergence for W12}).

The domain of the operator $L_{\nu,\ball_X}$ contains the space $Z^{2,2}_{U,N}(\ball_X,\nu)$, i.e. the completion of the space
\begin{gather*}
\mathcal{Z}(\ball_X)=\set{u\in\fcon^2_b(\ball_X)\tc \gen{\nabla_H u,\nabla_H G}_H=0\text{ for }\rho\text{-a.e. in }x\in \sfera_X};
\end{gather*}
with respect to the norm
\begin{align*}
\norm{u}^2_{Z^{2,2}_U(\ball_X,\nu)}=\norm{u}^2_{W^{2,2}(\ball_X,\nu)}+&\int_{\ball_X}\gen{\nabla_H^2U,\nabla_H u,\nabla_Hu}_Hd\nu+\\
&+2\sum_{n,m=1}^{+\infty}\lambda_n\delta_{n,m}\int_{\sfera_X}\partial_nu\partial_m u\frac{e^{-\Phi(1)}}{\norm{Q^{1/2}x}}d\rho.
\end{align*}
We want to show that in this case the space $Z^{2,2}_{U,N}(\Omega,\nu)$ only contains the constant functions. Indeed let $u\in \mathcal{Z}(\ball_X)$, without loss of generality assume that $u(x)=\varphi((x,h_1)_X,(x,h_2)_X)$ with $\varphi\in\con^2_b(\R^2)$. The Neumann boundary condition
\[\sum_{n=1}^{+\infty}\sqrt{\lambda_i}(x,h_i)_X\partial_i u(x)=0\text{ for }\rho\text{-a.e. in }x\in \sfera_X\]
implies
\[\sqrt{\lambda_1}(x,h_1)_X\partial_1 \varphi((x,h_1)_X,(x,h_2)_X)+\sqrt{\lambda_2}(x,h_2)_X\partial_2 \varphi((x,h_1)_X,(x,h_2)_X)=0 \]
for $\rho$-a.e $x\in \sfera_X$. So the function $\varphi$ satisfies the differential equation
\begin{gather}\label{NoIdea}
\sqrt{\lambda_1}\xi_1\partial_1\varphi(\xi)+\sqrt{\lambda_2}\xi_2\partial_2\varphi(\xi)=0\text{ for every }\xi\in\ball_{\R^2}.
\end{gather}
We want to remak that the condition $\xi\in\ball_{\R^2}$ is a consequence of the fact that, if $x\in\sfera_X$, then the vector $(\xi_1,\xi_2)=((x,h_1),(x,h_2))$ belongs to the unit ball of $\R^2$.
All the solutions of \eqref{NoIdea} are functions of the form 
\[\varphi(\xi)=g\pa{\xi_1^{\sqrt{\lambda_2}}\xi_2^{-\sqrt{\lambda_1}}},\] 
where $g$ is a sufficiently regular function in $\R$. It is easy to see that if $\varphi$ is non-constant, then $\varphi$ cannot be continuous at the origin.

So Theorem \ref{Main theorem} only gives us
\[D(L_{\ball_X,\nu})\subseteq W^{2,2}_{U,N}(\ball_X,\nu).\]
We want to remark that a positive answer to the question ``Is $\ball_X$ a Neumann extension domain?'' would allow us to apply Theorem \ref{thm extension} and get a characterization of the domain of $L_{\ball_X,\nu}$.

\subsection{The Ornstein--Uhlenbeck operator on half-spaces}\label{The Ornstein--Uhlenbeck operator on half-spaces}

In this section we give a characterization of the domain of the operator $L_{\mu,\Omega}$, where $\Omega$ is a half-space and $\mu$ is a centered non-degenerate Gaussian measure on a separable Banach space $X$. To do so we need some preliminary results, in particular a lemma about extensions of Sobolev functions and a proposition about finite dimensional approximations. We recall that $Z_{0}^{2,2}(X,\mu)=W^{2,2}(X,\mu)$ (see \cite{Bog98}). 

Let $x^*\in X^*\ssm\set{0}$ and $r\in \R$, throughout this section we set $G(x):=x^*(x)-r$ and $\Omega:=G^{-1}(-\infty,0]$. We recall that $x^*$ is a linear and continuous functional on $H$, so there exists $h_{x^*}\in H$ such that for every $h\in H$
\[x^*(h)=\gen{h_{x^*},h}_H.\]
Finally we remind the reader that 
\[W^{2,2}_{0,N}(\Omega,\mu)=\set{u\in W^{2,2}(\Omega,\mu)\tc x^{*}\pa{\trace(\nabla_H u)(x)}=0\text{ for }\rho\text{-a.e. }x\in G^{-1}(0)}.\]

\begin{lemma}\label{Extension}
There exists a Neumann extension operator from $W_{0,N}^{2,2}(\Omega,\mu)$ to $W^{2,2}(X,\mu)$.
\end{lemma}

\begin{proof}
We use a generalization of the reflection method, adapted to our Gaussian measure. Let $f\in\fcon_b^2(\Omega)$ and put
\begin{gather}\label{telefonocasa}
Ef(x):=\eqsys{f(x), & G(x)\leq 0,\\
\sum_{j=1}^7 a_jf\pa{x-(j+1)G(x)\frac{h_{x^*}}{\abs{h_{x^*}}_H^2}}\exp\pa{-\frac{c_{j}G(x)+b_jG^2(x)}{2\abs{h_{x^*}}_H}}, & G(x)>0;}
\end{gather}
where for every $j=1,\ldots,7$,
\begin{gather}
\label{telefono3}b_j=1-\frac{1}{j^2},\qquad\qquad c_j=\frac{2(j+1)r}{j^2}\pa{2-\frac{1}{j^2}};
\end{gather}
and
\begin{gather}
\label{telefono1}\sum_{j=1}^7 a_j=1,\qquad \sum_{j=1}^7a_j(j+1)=0,\qquad \sum_{j=1}^7 a_j(j+1)^2=0,\\
\label{telefono2}\sum_{j=1}^7a_jc_j(j+1)=0,\qquad \sum_{j=1}^7a_jc_j^2=0,\qquad \sum_{j=1}^7 a_jc_j=0,\qquad \sum_{j=1}^7a_jb_j=0.
\end{gather}
We start by proving that $Ef$ is well defined. Indeed for $j=1,\ldots,7$ and $x\in X$ such that $G(x)>0$ we have
\begin{gather*}
G\pa{x-(j+1)G(x)\frac{h_{x^*}}{\abs{h_{x^*}}_H^2}}=x^*\pa{x-(j+1)(x^*(x)-r)\frac{h_{x^*}}{\abs{h_{x^*}}_H^2}}-r=\\
=x^*(x)-(j+1)(x^*(x)-r)\frac{x^*(h_{x^*})}{\abs{h_{x^*}}_H^2}-r=-jG(x)\leq 0.
\end{gather*}
We point out that \eqref{telefono1} are the classical conditions to prove the continuity of $Ef$ and its derivatives. \eqref{telefono3} and \eqref{telefono2} arise from the exponential term in \eqref{telefonocasa}, which is used to prove the continuity estimate for the extension operator.

The fact that $Ef$ belongs to $\fcon_b^0(X)$ is obvious. Fix an orthonormal basis $\{h_i\}_{i\in\N}$ of $H$ obtained by completing the set $\{h_{x^*}/\abs{h_{x^*}}_H\}$, without loss of generality we let $h_1=h_{x^*}/\abs{h_{x^*}}$. Let $x_0\in X$ such that
$G(x_0)=0$, then $G(x_0+th_i)=t\delta_{1i}\abs{h_{x^*}}_H$. We have for $i\neq 1$
\[\partial_iEf(x_0)=\partial_if(x_0),\]
while
\begin{gather*}
\lim_{t\ra 0^-}\frac{1}{t}\pa{Ef(x_0+th_1)-Ef(x_0)}=\partial_1 f(x_0);\\
\lim_{t\ra 0^+}\frac{1}{t}\pa{Ef(x_0+th_1)-Ef(x_0)}
=\sum_{j=1}^7a_j \pa{\partial_1 f(x_0)\pa{1-(j+1)}-f(x_0)\frac{c_j}{2}}=\partial_1 f(x_0).
\end{gather*}
Thus, letting $T_j(x):=x-(j+1)G(x)\frac{h_{x^*}}{\abs{h_{x^*}}_H^2}$ and $A_j(x):=\exp\pa{-\frac{c_{j}G(x)+b_jG^2(x)}{2\abs{h_{x^*}}_H}}$
\begin{gather*}
\partial_i Ef(x)=\eqsys{\partial_i f(x), & G(x)\leq 0;\\
\sum_{j=1}^7a_jA_j(x)\pa{\pa{1-(j+1)\delta_{1i}}\partial_i f\pa{T_j(x)}-f\pa{T_j(x)}\frac{2b_jG(x)+c_j}{2}\delta_{1i}}, & G(x)>0.}
\end{gather*}
In the same way it holds
\begin{gather*}
\partial_j\partial_i Ef(x)=\eqsys{\partial_j\partial_i f(x), & G(x)\leq 0;\\
B_{ij}(x), & G(x)>0,}
\end{gather*}
where
\begin{gather*}
B_{ij}(x):=\sum_{l=1}^7a_l \partial_{ij}f\pa{x-(l+1)G(x)\frac{h_{x^*}}{\abs{h_{x^*}}^2}}e^\pa{-\frac{c_{l}G(x)+b_lG^2(x)}{2\abs{h_{x^*}}_H}}\pa{1-(l+1)\delta_{1i}}\pa{1-(l+1)\delta_{1j}}+\\
-\sum_{l=1}^7a_l \partial_{i}f\pa{x-(l+1)G(x)\frac{h_{x^*}}{\abs{h_{x^*}}^2}}e^\pa{-\frac{c_{l}G(x)+b_lG^2(x)}{2\abs{h_{x^*}}_H}}\pa{1-(l+1)\delta_{1i}}\frac{2b_lG(x)+c_l}{2}\delta_{1j}+\\
-\sum_{l=1}^7a_l \partial_{j}f\pa{x-(l+1)G(x)\frac{h_{x^*}}{\abs{h_{x^*}}^2}}e^\pa{-\frac{c_{l}G(x)+b_lG^2(x)}{2\abs{h_{x^*}}_H}}\pa{1-(l+1)\delta_{1j}}\frac{2b_lG(x)+c_l}{2}\delta_{1i}+\\
+\sum_{l=1}^7a_l f\pa{x-(l+1)G(x)\frac{h_{x^*}}{\abs{h_{x^*}}^2}}e^\pa{-\frac{c_{l}G(x)+b_lG^2(x)}{2\abs{h_{x^*}}_H}}\frac{(2b_lG(x)+c_l)^2}{4}\delta_{1i}\delta_{1j}+\\
+\abs{h_{x^*}}_H\sum_{l=1}^7a_lb_lf\pa{x-(l+1)G(x)\frac{h_{x^*}}{\abs{h_{x^*}}^2}}e^\pa{-\frac{c_{l}G(x)+b_lG^2(x)}{2\abs{h_{x^*}}_H}}\delta_{1i}\delta_{1j}.
\end{gather*}
So $Ef$ belongs to $\fcon_b^2(X)$ and $Ef(x)=f(x)$, $\nabla_HEf(x)=\nabla_Hf(x)$, $\nabla_H^2Ef(x)=\nabla_H^2f(x)$ for every $x\in\Omega$. Without loss of generality we can assume that there exists $n\in\N$ and $\psi\in\con^2_b(\R^{n})$ such that for every $x\in X$
\[f(x)=\psi(x^*(x),\hat{h}_2(x),\ldots,\hat{h}_n(x)).\]
We remark that
\[Ef(x)=\eqsys{\psi(x^*(x),\hat{h}_2(x),\ldots,\hat{h}_n(x)), & x^*(x)\leq r,\\
\sum_{j=1}^7a_j\psi\pa{-jx^*(x)+(j+1)r,\hat{h}_2(x),\ldots,\hat{h}_n(x)}\exp\pa{-\frac{c_{j}(x^*(x)-r)+b_j(x^*(x)-r)^2}{2\abs{h_{x^*}}_H}}, & x^*(x)> r.}\]
So we have
\begin{gather}
\int_X\abs{Ef(x)}^2d\mu(x)\leq\int_{\xi_1\leq r}\abs{\psi(\xi_1,\xi_2,\ldots,\xi_n)}^2d\mu_n(\xi)+\notag\\
+7\sum_{j=1}^7a_j^2\int_{\xi_1> r}\abs{\psi\pa{-j\xi_1+(j+1)r,\xi_2,\ldots,\xi_n}\exp\pa{-\frac{c_{j}(\xi_1-r)+b_j(\xi_1-r)^2}{2\abs{h_{x^*}}_H}}}^2d\mu_n(\xi).\label{finiamolaqui}
\end{gather}
We remark that $d\mu_n(\xi)=\exp(-\abs{\xi}^2/2\abs{h_{x^*}}_H)dx$. For every $j=1,\ldots,7$, consider the change of variable:
\begin{gather}\label{cov}
\eqsys{\eta_1=-j\xi_1+(j+1)r;\\
\eta_i=\xi_i, & i=2,\ldots,7.}
\end{gather}
We use \eqref{cov} in the second integral of \eqref{finiamolaqui}, and we get
\begin{gather*}
\sum_{j=1}^7a_j^2\int_{\xi_1> r}\abs{\psi\pa{-j\xi_1+(j+1)r,\xi_2,\ldots,\xi_n}\exp\pa{-\frac{c_{j}(\xi_1-r)+b_j(\xi_1-r)^2}{2\abs{h_{x^*}}_H}}}^2d\mu_n(\xi)=\\
=\sum_{j=1}^7a_j^2\int_{\xi_1> r}\abs{\psi\pa{-j\xi_1+(j+1)r,\xi_2,\ldots,\xi_n}\exp\pa{-\frac{c_{j}(\xi_1-r)+b_j(\xi_1-r)^2}{2\abs{h_{x^*}}_H}}^2}e^{-\frac{\abs{\xi}^2}{2\abs{h_{x^*}}_H}}d\xi=\\
=\sum_{j=1}^7a_j^2\int_{\eta_1\leq r}\abs{\psi\pa{\eta_1,\eta_2,\ldots,\eta_n}\exp\pa{-\frac{c_{j}\pa{-\frac{\eta_1-r}{j}}+b_j\pa{-\frac{\eta_1-r}{j}}^2}{2\abs{h_{x^*}}_H}}}^2e^{-\frac{\pa{-\frac{\eta_1-(j+1)r}{j}}^2+\sum_{i=2}^n\eta_i^2}{2\abs{h_{x^*}}_H}}d\xi.
\end{gather*}
Using the definition of $a_j,b_j$ and $c_j$ we get
\begin{gather*}
\sum_{j=1}^7a_j^2\int_{\xi_1> r}\abs{\psi\pa{-j\xi_1+(j+1)r,\xi_2,\ldots,\xi_n}\exp\pa{-\frac{c_{j}(\xi_1-r)+b_j(\xi_1-r)^2}{2\abs{h_{x^*}}_H}}}^2d\mu_n(\xi)=\\
=C \int_{\xi_1\leq r}\abs{\psi(\xi_1,\xi_2,\ldots,\xi_n)}^2d\mu_n(\xi)
\end{gather*}
for some constant $C>0$. So
\begin{gather*}
\int_X\abs{Ef(x)}^2d\mu(x)\leq K\int_{\xi_1\leq r}\abs{\psi(\xi_1,\xi_2,\ldots,\xi_n)}^2d\mu_n(\xi)=K\int_\Omega \abs{f(x)}^2d\mu(x),
\end{gather*}
where the constant $K>0$ depend only on $r$ and $a_i$ for $i=1,\ldots,7$. Using similar arguments on $\nabla_H Ef$ and $\nabla_H^2 Ef$ we get for every $f\in\fcon^2_b(\Omega)$
\[\norm{Ef}_{W^{2,2}(X,\mu)}\leq \tilde K\norm{f}_{W^{2,2}(\Omega,\mu)},\]
where $\tilde K>0$ is an adequate constant independent of $f$. A standard denstity argument gives the thesis of our lemma.
\end{proof}
Using Lemma \ref{Extension} and Theorem \ref{thm extension} we get a characterization of the domain of $L_{\mu,\Omega}$. In order to get Theorem \ref{thm halfspaces} we need a further approximation argument.

\begin{pro}\label{approx in halfspaces}
Let $u\in W^{2,2}(\Omega,\mu)$ be such that $\gen{\nabla_H u(x),h_{x^*}}_H=0$ for $\rho$-a.e. $x\in G^{-1}(0)$. There exists a sequence $(u_n)_{n\in\N}$ belonging to $\fcon^2_b(\Omega)$ such that
\begin{enumerate}
\item $\gen{\nabla_H u_n(x),h_{x^*}}_H=0$ for every $n\in\N$ and $\rho$-a.e. $x\in G^{-1}(0)$;\label{ebbasta1}

\item $(u_n)_{n\in\N}$ converges to $u$ in $W^{2,2}(\Omega,\mu)$.\label{ebbasta2}
\end{enumerate}
\end{pro}

\begin{proof}
Fix an orthonormal basis $\{h_i\}_{i\in\N}$ of $H$ obtained by completing the set $\{h_{x^*}/\abs{h_{x^*}}_H\}$, without loss of generality we let $h_1=h_{x^*}/\abs{h_{x^*}}_H$. Let $u\in W^{2,2}(\Omega,\mu)$ be such that 
\begin{gather}\label{stufoooo}
\partial_1u(x)=\gen{\nabla_H u(x),h_{x^*}}_H=0\text{ for $\rho$-a.e. $x\in G^{-1}(0)$}
\end{gather} 
Let $Eu$ be the extension defined in Lemma \ref{Extension}. We denote with $P_n,S_n:X\ra H$ the functions defined as
\[P_nx=\sum_{i=1}^n\hat{h}_i(x)h_i\quad\text{ and }\quad S_ny=\sum_{i=n+1}^{+\infty}\hat{h}_i(y)h_i;\]
for every $x,y\in X$. We recall that $P_nx$ converges pointwise $\mu$-a.e. $x\in X$ to $x$ (see \cite[Theorem 3.5.1]{Bog98}). Let
\[v_n(x)=\int_XEu(P_nx+S_ny)d\mu(y),\]
by \cite[Corollary 3.5.2 and Proposition 5.4.5]{Bog98} $u_n$ converges to $Eu$ in $W^{2,2}(X,\mu)$ as $n$ goes to infinity and for every $i,n\in\N$
\[\partial_iv_n(x)=\eqsys{\int_X\partial_iEu(P_nx+S_ny)d\mu(y) & i\leq n\\ 0 & i>n}\]
Observe that if $x\in G^{-1}(0)$, then for every $y\in X$ and $n\in\N$
\begin{gather*}
G(P_nx+S_ny)=x^*(P_nx+S_ny)-r=\sum_{i=1}^n\hat{h}_i(x)\gen{h_{x^*},h_i}_H+\sum_{i=n+1}^{+\infty}\hat{h}_i(y)\gen{h_{x^*},h_i}_H-r=\\
=\hat{h}_1(x)\abs{h_{x^*}}_H-r=x^*(x)-r=0.
\end{gather*}
By \eqref{stufoooo} we get
\[\gen{\nabla_H v_n(x),h_{x^*}}_H=\int_X\partial_1Eu(P_nx+S_ny)d\mu(y)=\int_X\partial_1u(P_nx+S_ny)d\mu(y)=0,\] 
for $\rho$-a.e. $x\in G^{-1}(0)$.

We are almost done, but we need smoother function satisfying Proposition \ref{approx in halfspaces}\eqref{ebbasta1}-\eqref{ebbasta2}. Let $\psi_{n}(\xi):=v_n(\sum_{i=1}^n\xi_ih_i)$. We remind the reader that $\psi_{n}$ belongs to $W^{2,2}(\R^n,\mu\circ P_n^{-1})$ and
\[\partial_1 \psi_{n}(\xi)=0\text{ for }\xi\in\R^n\text{ such that }\xi_1=r.\]
Let $\mathcal{L}_0^{n}$ be the generator of the $m$-dimensional Ornstein--Uhlenbeck operator with homogeneous Neumann condition in $\elle^2(\mathcal{O}_n,\mu\circ P_n^{-1})$, where $\mathcal{O}_n=\set{\xi\in\R^n\tc \xi_1\leq r}$. By \cite[Theorem 12.4.9]{BL07} we know that the domain of $\mathcal{L}^n_0$ in $\elle^2(\mathcal{O}_n,\mu\circ P_n^{-1})$ is
\[D(\mathcal{L}^n_0)=\set{\varphi\in W^{2,2}(\mathcal{O}_n,\mu\circ P_n^{-1})\tc \gen{\xi,\nabla \varphi}\in\elle^2(\mathcal{O}_n,\mu\circ P_n^{-1}),\ \partial_1\varphi(\xi)=0\text{ when }\xi_1=r}\]
and
\begin{gather*}
\norm{D^kR(\lambda,\mathcal{L}_0^n)}_{L(\elle^2(\mathcal{O}_n,\mu\circ P_n^{-1}))}\leq 2^k\lambda^{\frac{k}{2}-1},
\end{gather*}
where $R(\cdot,\mathcal{L}_0^n)$ is the resolvent operator associate to $\mathcal{L}_0^n$ and $k=0,1,2$. Let $f_n:=\psi_{n}-\mathcal{L}_0^n\psi_{n}$, where the equality is meant in $\elle^2(\mathcal{O}_n,\mu\circ P_n^{-1})$. Let $(f_{n,k})_{k\in\N}$ be a sequence of bounded smooth function such that $f_{n,k}$ converges in $\elle^2(\mathcal{O}_n,\mu\circ P_n^{-1})$ to $f_{n}$ as $k$ goes to infinity. We let
\[\psi_{n,k}=R(1,\mathcal{L}_0^n)f_{n,k}.\]
We recall that $\psi_{n,k}$ belongs to $D(\mathcal{L}_n^0)$ and to $\con^2_b(\R^n)$ (see \cite[Section 12]{BL07}). Let
\[F_{n,k}u(x):=\psi_{n,k}(\hat{h}_1(x),\ldots,\hat{h}_n(x)).\]
We get that $F_{n,k}u$ belongs to $\fcon_b^2(\Omega)$ and satisfy the Neumann condition at the boundary. Let $\eps>0$ and consider $n_\eps,k_\eps\in\N$ such that
\begin{gather*}
\norm{v_{n_\eps}-Eu}_{W^{2,2}(X,\mu)}\leq \frac{\eps}{2};\qquad
\norm{\psi_{n_\eps,k_\eps}-\psi_{n_\eps}}_{W^{2,2}(\mathcal{O}_{n_\eps},\mu\circ P^{-1}_{n_\eps})}\leq\frac{\eps}{2}
\end{gather*}
So
\begin{gather*}
\norm{F_{n_\eps,k_\eps}u-u}_{W^{2,2}(\Omega,\mu)}\leq \norm{F_{n_\eps,k_\eps}u-Eu}_{W^{2,2}(X,\mu)}\leq\\ \norm{F_{n_\eps,k_\eps}u-v_{n_\eps}}_{W^{2,2}(X,\mu)} +\norm{v_{n_\eps}-Eu}_{W^{2,2}(X,\mu)}\leq\\ \leq \norm{\psi_{n_\eps,k_\eps}-\psi_{n_\eps}}_{W^{2,2}(\mathcal{O}_{n_\eps},\mu\circ P^{-1}_{n_\eps})}+\norm{v_{n_\eps}-Eu}_{W^{2,2}(X,\mu)}\leq \eps.
\end{gather*}
Thus the sequence $u_{m}:=F_{n_{m^{-1}},k_{m^{-1}}}u$ for $m\in\N$ is the sequence we were lookong for.
\end{proof}

As a consequence of Corollary \ref{cor halfspaces} and Proposition \ref{approx in halfspaces}, we get Theorem \ref{thm halfspaces}.

\begin{ack}
The authors would like to thank Prof. Alessandra Lunardi, Prof. Diego Pallara and Prof. Leonardo Biliotti for many useful discussions and comments. The authors are members of GNAMPA of the Italian Istituto Nazionale di Alta Matematica (INdAM). 

This research was partially supported by the PRIN 2015 grant: ``Deterministic and stochastic evolution equations'' and the GNAMPA 2017 project: ``Equazioni e sistemi di equazioni di Kolmogorov in dimensione finita e non''.
\end{ack}

\bibliographystyle{plain}
\nocite{*} 
\bibliography{bibpesi}
\markboth{\textsc{References}}{\textsc{References}}

\end{document}